\documentclass{aims1}
\usepackage{amsmath}
\usepackage{alltt}
\usepackage{exscale}
\usepackage{amsfonts}
\usepackage{array}
\usepackage{latexsym}
\usepackage{enumerate}
\usepackage{graphicx}
\usepackage{tikz}
\usepackage{color}
  \usepackage{paralist}
  \usepackage[T1]{fontenc}
  \usepackage{mathtools}
  \usepackage{graphics} 
  \usepackage{epsfig} 
\usepackage{graphicx} 
 \usepackage{epstopdf}
 \usepackage[colorlinks=true]{hyperref}
   
   \usepackage[english]{babel}
\usepackage[latin1]{inputenc}
\usepackage{amsthm,amsfonts,amssymb}
\usepackage{pgf,pgfarrows,pgfautomata,pgfheaps,pgfnodes,pgfshade}
\usepackage{times}
\hypersetup{urlcolor=blue, citecolor=red}

     \def\DOI{10.3934/xx.xx.xx.xx}

\newtheorem{theorem}{Theorem}[section]

\theoremstyle{definition}
\newtheorem{definition}[theorem]{Definition}
\newtheorem{remark}{Remark}

\usepackage{mathtools}
\newcommand{\fracj}{\displaystyle \frac}

\newcommand{\limj}{\displaystyle \lim}

\newcommand{\intj}{\displaystyle \int}

\newcommand{\minj}{\displaystyle \min}

\newcommand{\beq}{\begin{equation}}
\newcommand{\eeq}{\end{equation}}

\def\a#1{{\mathbb #1}}
\title[A strategy to reduce variations in an exact control] 
      {Total variations reduction for  an exact control applied to a dynamical mechanical system}

\author[]{Philippe Destuynder and Erwan Liberge}
\date{09-23-2024}

\begin{document}

\maketitle
%
%
{\footnotesize
   \centerline{LaSIE - UMR CNRS 7356, La Rochelle University}
   \centerline{  Avenue Michel Cr\'epeau, La Rochelle, 17000, France}
   \centerline{ \email{philippe.destuynder@univ-lr.fr} and \email{erwan.liberge@univ-lr.fr}}

} 


\maketitle
\begin{abstract} In an optimal control strategy, an important point is to define the cost of the control. Usually it is added to the control criterion and multiplied by a small coefficient denoted by $\varepsilon$ which is known as the marginal cost of the control. The key idea of this paper, is to introduce a smoothing term in the control cost which aims at reducing the quantity of energy spent and reducing the oscillations of the control. Then using a so-called asymptotic control based on the smallness of $\varepsilon$, we construct an exact control which can be implemented in a close loop.  The energy involved in the control depends mainly on the variation of the control. Therefore it seems natural to include this quantity (the total variations) in the criterion involved in the optimal control. This can be done approximately by introducing the $L^1$ norm of the first order derivative of the control. The control strategy that we develop in this paper can be applied to such linear models. One important and new point is that we focus on exact control strategies for a non-differentiable criterion because of the cost of the control. Following the ideas of Tykhonov regularization method, it is proved using the so-called asymptotic method based on the smallness of the marginal cost of the control, that the exact control suggested is the one which represents the minimum of the marginal cost among exact controls. Furthermore, and it is the main technical point, it can reduce the variations of the control with an adequate tuning of the various parameters of the control loop.
We test the method on three examples. The first one concerns an elementary case just to show that the algorithm works on a reference situation. The second one is the control of flight of a flying boat with foils. We describe in the appendix the mechanical model which can be unstable because of a negative damping. The third example concerns the America's cup boat Luna Rossa.
\end{abstract}
%
%
\section{Introduction} There are a lot of contributions for control algorithms applied to dynamics of mechanical devices. But most of them does not focus on the control function. A large number of them makes use of the optimal control strategy and consider a deterministic approach. The control is mostly a square integrable function versus the time variable and can be discontinuous (bang-bang when the control time is small and the amplitude restricted). Obviously this approach is only well founded  if an exact controllability is satisfied for the state variables implied in the control criterion.\\ In our formulation there are two important differences: \\ 1 first of all we introduce a norm of  the total variation of the control; \\ 2 secondly we make tend to zero the marginal cost of the control in order to obtain an exact control with a minimal cost among the exact controls as far as the exact controllability is satisfied. In fact this is similar to Tykhonov regularization method \cite{TYKHONOV1}-\cite{TYKHONOV2}.
\\
The main difficulty when one introduces a term which takes into account the total variation (even approximately as we do) is that the derivability of the criterion is lost. The difficulty is similar to the minimization of the energy for non Newtonian fluid. This is the case of most gels like the ketchup or tooth pastes. Roland Glowinski  \cite{GLOW} developed fifty years ago a special strategy for solving these models. It is based on a duality method which enables to transform most of second type variational inequalities (following the terminology of G. Stampacchia \cite{STAMP}) into a derivable optimization problem but with constraints on the variables. The Glowinski method has already being applied to image processing in several publications and the algorithm and the numerical results obtained by Hela Sellami \cite{SELLAMI} confirmed the efficiency of the extension to image processing. This is what we do in this text but this time for an exact  control problem.\\
 In a first part we formulate the general mechanical model to be controlled. The second part is devoted to the control algorithm. The three following sections contain numerical examples. 
First of all we test a very simple mechanical system with two masses and two springs. Then we discuss the algorithm for  a flying sailing boat that we have invented and finally we apply the algorithm to the America's cup boat Luna Rossa \cite{LUNA}. For this last example we took the data given in the article.
\section{The dynamical model to be controlled}  We consider a mechanical sytem in dynamics with $N\geq 1$ degrees of freedom.  The movement is parametrized by the vector $X(t)\in \a R^N$ and its time derivative (the velocity) is $\dot X(t)$. The acceleration is denoted by $\ddot X(t)$ (we are working in a Galilean frame). We introduce four $N\times N$ matrices in order to define the dynamical model. One denoted by ${\mathcal M}$, is the inertia, the second one ${\mathcal C}$ is the physical damping (not necessarily positive), the third one -say ${\mathcal K}$- is the stiffness and the last one ${\mathcal B}$ is the control tuner. The movements (displacements and rotations) are expressed in a Galilean system of axis moving at constant velocity -say $V$ in the direction $-x_1$ for instance. 
 The vector of the degrees of freedom is $X\in \a R^N$. It is a function of time. 
  The one representing the control is $u\in \a R^p,\;p\leq N$,  
 and the external perturbations applied to the system are represented by the vector ${\mathcal F}\in \a R^N$  
 which components can also depend on time. There are also initial condition on $X(0)=X_0\in \a R^N$ and $\dot X(0)=X_1\in \a R^N$.\\
 The model that we consider is linear but it is derived from a non linear one by a linearization, mainly because of the dependance of the external forces on $X$ and $\dot X$. Several instabilities can occur like a negative stiffness (catastrophic), a multiplicity of eigenfrequencies (flutter) or a buffting if one consider that the stiffness can depend on time for instance. The velocity acts on the damping matrix and a negative damping can occur (similar to a stall flutter phenomenon). Hence many improvements could be discuss on this very simple but nevertheless realistic model. In the examples that we treat in the following only damping instabilities occurs (case of the flying boats). 
 Then the dynamical model for a control  $\delta$ consists in solving the following ordinary second order differential equation which can be solved numerically using for instance a Newmark time steps scheme \cite{NEWMARK}:
 \\\\
 \beq\label{eq110}{\mathcal M}\ddot X+{\mathcal C}\dot X+{\mathcal K}X={\mathcal F}+{\mathcal B}u,\;X(0)=X_0,\;\;\dot X(0)=X_1.\eeq
 \\
 \begin{remark}\label{rem1}There is another possibility for solving (\ref{eq110}) which can be particularly interesting for a real time implementation. We introduce the extended vector of $\a R^{2N}$:\\
 \beq\label{eq120}Z=\left (\begin{array}{l} \dot X \\\\ X\end{array}\right )\;\;Z_0=Z(0)=\left (\begin{array}{l} X_1 \\\\ X_0\end{array}\right )\eeq
 \\
 and a new matrix ${\mathcal A}$ and a vector ${\mathcal Q}$ by:
 \\
 \beq\label{eq130}\begin{array}{cc}{\mathcal A}=\left (\begin{array}{l}\begin{array}{cc}0&\;\;\;\;\;\;\;\;\;\;\;\;\;\;\;\;I_d\end{array}\\\\\begin{array}{cc}-\hskip0.02cm {\mathcal M}^{-1}{\mathcal C}&-\hskip0.02cm {\mathcal M}^{-1}{\mathcal K}\end{array}\end{array}\right )&{\mathcal Q}=\left (\begin{array}{l}\;\;\;\;0\\\\ {\mathcal M}^{-1}{\mathcal F}+{\mathcal M}^{-1}{\mathcal B}u\end{array}\right )\end{array}
 \eeq
 \\
 The model (\ref{eq110}) is equivalent to the following one:\\
 \beq \label{eq140}\dot Z={\mathcal A}Z+{\mathcal Q},\;\;Z(0)=Z_0.\eeq
 The reduction of the matrix ${\mathcal A}$ enables to simplify the initial model. Setting ($S$ is the change of basis matrix for the diagonalisation  of the matrix ${\mathcal A}$ or its Jordan reduction if there are multiple eigenvalues):\\  \\$Y=SZ$ and $\Lambda=S{\mathcal A}S^{-1}$ which is a diagonal  (or Jordan) matrix,  \\\\the equation (\ref{eq110}) can also be written (${\mathcal U}=S{\mathcal Q}$):
 \beq\label{eq160}\dot Y=\Lambda Y+{\mathcal U},\;\;Y(0)=SZ_0.\eeq
 Each line (or group of lines associated to repeated eigenvalues) of this equation can be solved quasi-analytically but an integral approximation is still required  for the right-hand side.\hfill $\Box$
  \end{remark}
\section{The control problem} In this section we start with the optimal control strategy and then we focus on a solution method. Finally we introduce an asymptotic method versus the small parameter $\varepsilon$ which is the marginal cost of the control.  This enables to construct an exact control.
\subsection{Classical Optimal control strategy} Let us first recall the basic formulation of optimal control applied to a a linear second order system as the one described in the previous section.  The state function denoted by $X$ is a function of time on the interval $[0,T]$. The sollicitation of the system is represented by initial conditions $(X_0,X_1)\in \a R^{2N}$ and an external load ${\mathcal F}\in L^2(]0,T[; \a R^N)$. The state equation is written as follows where ${\mathcal M,C,K}$ are $N\times N$ matrices, ${\mathcal M}$ being symmetrical positive definite  and ${\mathcal B}$ a rectangular $N\times p$ matrix (where $p$ is the number of control functions):
\beq\label{eq150}\left \{ \begin{array}{l}\hbox{Find }X(t)\in \a R^N\hbox{ such that:}\\\\
{\mathcal M}\ddot X+{\mathcal C}\dot X+{\mathcal K}X={\mathcal F}+{\mathcal B} u,\;X(0)=X_0,\;\dot X(0)=X_1,
\end{array}\right.\eeq
For any $(\varepsilon,\alpha,\beta,\gamma)\in \a R^{+,*}\times \a R^{+}\times \a R^{+,*}\times \a R^{+,*}$  we define a control criterion function of a control $v\in [H^1_0(]0,T[)]^p$ by ($\vert.\vert_{\a R^m}$ is the norm in $\a R^m$ and 
$(.,.)$ is the scalar product):
\beq\label{eq170}J^{\varepsilon,\alpha,\beta,\gamma}(v)\hskip-.05cm=\hskip-.05cm\fracj{1}{2}{\Big (}\vert X(T)\vert_{\a R^{N}}^2\hskip-.05cm+\hskip-.05cm\vert \dot X(T)\vert_{\a R^N}^2\hskip-.05cm+\hskip-.1cm\varepsilon\int_0^T \hskip-.1cm{\Big [}\alpha\vert v \vert_{\a R^p}^2\hskip-.05cm+\hskip-.05cm\beta\vert \fracj{dv}{dt}\vert_{\a R^p}^2\hskip-.05cm+\hskip-.05cm2\gamma\vert \fracj{dv}{dt}\vert_{\a R^p} {\Big ]}{\Big )}.\eeq
The optimal control consists in minimizing $J^{\varepsilon,\alpha,\beta,\gamma}(v)$ versus $v\in {\Big[}H^1_0(]0,T[){\Big]}^p$.\\\\
The existence and uniqueness of a solution are quite classical as far as this is a classical optimization problem with a {\it coercive}, continuous and strictly convex functional to be minimized in a reflexive space: (${\Big [}H^1_0(]0,T[){\Big ]}^p$).\\\\
Unfortunately, the gradient of $J^{\varepsilon,\alpha,\beta,\gamma}$ does not exist if $\gamma\neq 0$.  Nevertheless a part of the criterion has a derivative and it can be computed  using an adjoint state -say $P$- solution of (the notation $\;^{t}$ denotes the transposition):\\
\beq\label{eq160}{\mathcal M}\ddot P-\;^{t}{\mathcal C}\dot P+\;^{t}{\mathcal K}P=0,\;{\mathcal M}P(T)=\dot X(T),\;\; {\mathcal M}\dot P(T)=-X(T)+\;\hskip-0.05cm^{t}{\mathcal C}P(T). \eeq
Hence the derivative of the criterion without the term with $\gamma$, at $u\in {\Big [}H^1_0(]0,T[){\Big ]}^p$ in the direction $v\in {\Big [}H^1_0(]0,T[){\Big ]}^p$  is:
\beq\label{eq170bis}\left \{\begin{array}{l}\limj_{\eta\rightarrow 0}\fracj{J^{\varepsilon,\alpha,\beta,0}(u+\eta v)-J^{\varepsilon,\alpha,\beta,0}(u)}{\eta}=G^{\varepsilon,\alpha,\gamma,0}(u)(v)\\\\=\intj_0^T\;^{t}{\mathcal B} P.v+\varepsilon{\Big [}\alpha (u, v)+\beta(\fracj{du}{dt},\fracj{dv}{dt}){\Big]}=\intj_0^T{\Big (}\;^{t}{\mathcal B}P+\varepsilon(\alpha u-\beta\fracj{d^2 u}{dt^2}),v{\Big )}\end{array}\right.\eeq
Because the full criterion, $J^{\varepsilon,\alpha,\beta;\gamma}$ is not derivable (due to the term in factor of $\gamma\neq 0$), we use a second type variational inequality  (following G. Stampacchia \cite{STAMP}) for characterizing the solution which minimizes the criterion $J^{\varepsilon,\alpha;\beta,\gamma}$ explicited at (\ref{eq170}).  Hence the optimality condition is:
\beq\label{eq180}\left \{\begin{array}{l}\hbox{ find }u^\varepsilon \in {\Big [}H^1_0(]0,T[){\Big ]}^p\;\hbox{ such that:}\\\\ 
\forall  v \in {\Big [}H^1_0(]0,T[){\Big ]}^p,\;\;G^{\varepsilon,\alpha,\beta,0}(u^\varepsilon)(v-u^\varepsilon)+\gamma\intj_0^T\vert v\vert_{\a R^p}- \vert u^\varepsilon\vert_{\a R^p} \geq 0.\end{array}\right.\eeq
Or else by expliciting the gradient $G^{\varepsilon,\alpha,\beta,0}$, where $P^\varepsilon$ is the solution of the adjoint state (\ref{eq160}) and with $X^\varepsilon$  solution of the primal equation (\ref{eq150}) with the control $u^\varepsilon$:
\beq\label{eq190}\left \{\begin{array}{l}\forall  v \in {\Big [}H^1_0(]0,T[){\Big ]}^p,\\\\ 
\intj_0^T{\Big (}\;^{t}{\mathcal B}P^\varepsilon+\varepsilon(\alpha u^\varepsilon-\beta\fracj{d^2 u^\varepsilon}{dt^2}),v-u^\varepsilon{\Big )}+\gamma\intj_0^T \vert v\vert_{\a R^p}- \vert u^\varepsilon\vert_{\a R^p} \geq 0.\end{array}\right.\eeq
  \begin{remark}\label{rem10}In fact the above writing  is an {\it abuse} because:
$$\fracj{d^2 u^\varepsilon}{dt^2}\in {\Big [}H^{-1}(]0,T[){\Big ]}^p.$$
 Hence the corresponding term should be written with  brackets: $\langle.,.\rangle$, for representing the duality between the spaces ${\Big [}H^{-1}(]0,T[){\Big ]}^p$ and ${\Big [}H_0^{1}(]0,T[){\Big ]}^p$. But it is more convenient to write this term as follows:
\beq\label{eq200} \intj_0^T\beta (\fracj{du^\varepsilon}{dt},\fracj{dv}{dt}-\fracj{du^\varepsilon}{dt}).\eeq
\;\hfill $\Box$
\end{remark}

\subsection{The idea of R. Glowinski applied to the control model (when $\gamma>0$)} R. Glowinski \cite{GLOW} has suggested a new idea for non-Newtonian fluids as the Bingham model (Ketchup, mayonnaise, toothpaste...) and more generally gels. His method consists in two steps: the first one is a regularization of the non-differentiable term using a small parameter and the second step is a limit analysis when the regularization parameter tends to $0$.\\
\subsubsection{The regularization of the non derivable term.} \label{step1} First of all we construct another almost equivalent model for solving (\ref{eq190}) once $P^\varepsilon$ is known.  For any $\eta>0$ we introduce  a regularization of the norm of $\fracj{du}{dt}$, by:
$$\intj_0^T\vert \fracj{du}{dt}\vert_{\a R^p}\simeq \intj_0^T\sqrt{ \eta+\vert\fracj{du}{dt}\vert_{\a R^p}^2}-\sqrt{\eta},$$
which is G\^ateaux derivable at $u\in {\Big [}H^1_0(]0,T[){\Big ]}^p$ in the direction $v\in {\Big [}H^1_0(]0,T[){\Big ]}^p$. This derivative  is:
$$\intj_0^T\frac{(\frac{du}{dt},\frac{dv}{dt})}{\sqrt{\eta+\vert\frac{du}{dt}\vert_{\a R^p}^2}}.$$
 and an approximation of the  problem which characterizes $u^\varepsilon$ is (let us recall that $P^\varepsilon$ is assumed to be  given at this step):
\beq\label{eq201}\hskip-.1cm\left\{ \begin{array}{l}\hbox{find }u^{\varepsilon,\eta}\in {\big[}H^1_0(]0,T[){\big ]}^p\hbox{ which minimizes the quantity:}
\\\\
\fracj{\varepsilon}{2}\intj_0^T{\Big [}\alpha\vert v\vert_{\a R^p}^2+\beta\vert \fracj{dv}{dt}\vert_{\a R^p}^2+2\gamma(\sqrt{\eta+\vert \fracj{dv}{dt}\vert_{\a R^p}^2}-\sqrt{\eta}){\Big ]}+\intj_0^T(\;^{t}{\mathcal B}P^\varepsilon, v).
\end{array}\right.\eeq
Here again the functional to be minimized is continuous {\it coercive} and strictly convex. Therefore there is a unique solution to the previous problem.  But now, because the new functional is derivable, one can characterize the solution 
$u^{\varepsilon,\eta}\in {\Big [}H^1_0(]0,T[){\Big ]}^p$ with an equation:
\beq\label{eq202}\left \{\begin{array}{l}\forall v\in {\Big [}H^1_0(]0,T[){\Big ]}^p:\\\\
\intj_0^T{\Big [}\alpha (u^{\varepsilon,\eta},v)+ \beta (\fracj{du^{\varepsilon,\eta}}{dt},\fracj{dv}{dt}) +\gamma\frac{(\frac{du^{\varepsilon,\eta}}{dt},\frac{dv}{dt})}{\sqrt{\eta+\vert\frac{du^{\varepsilon,\eta}}{dt}\vert_{\a R^p}^2}}+\fracj{1}{\varepsilon}(\;^{t}{\mathcal B}P^{\varepsilon},v){\Big ]}=0,
\end{array}\right.
\eeq
where $P^{\varepsilon}$ is the adjoint state solution of (\ref{eq160})  (hence does not depends on $\eta$) solution of the adjoint equation associated to the state variable $X^{\varepsilon}$ solution of the state equation with the control $u^{\varepsilon}$.
It is convenient for the following to set: 
\beq\label{eq300}\lambda^{\varepsilon,\eta}=\frac{\frac{du^{\varepsilon,\eta}}{dt}}{\sqrt{\eta+\vert\frac{du^{\varepsilon,\eta}}{dt}\vert_{\a R^p}^2}}\in {\Big [}L^2(]0,T[){\Big ]}^p,\;\hbox{ and }\vert\lambda^{\eta}\vert_{\a R^p}\leq 1.\eeq
Furthermore one has\footnote{$\forall x\in \a R^+,\;\fracj{x^2}{\sqrt{\eta+x^2}}=\fracj{\eta+x^2}{\sqrt{\eta+x^2}}-\fracj{\eta}{\sqrt{\eta+x^2}}\geq x-\sqrt{\eta}$}:
\beq\label{eq310}\vert\lambda^{\varepsilon,\eta}\vert_{\a R^p}\leq 1\hbox{ and }\vert\fracj{du^{\varepsilon,\eta}}{dt}\vert_{\a R^p}-\sqrt{\eta}\leq (\lambda^{\varepsilon,\eta},\fracj{du^{\varepsilon,\eta}}{dt})\leq \vert\fracj{du^{\varepsilon,\eta}}{dt}\vert_{\a R^p}.\eeq
Setting:
 \beq\label{eq311}B_1=\{\mu\in \a R^p,\;\vert \mu\vert_{\a R^p} \leq 1\},\eeq
 it is worth to point out that one has from (\ref{eq300}):
\\
\beq\label{eq210}\forall \mu\in B_1,\;(\mu,\fracj{du^{\varepsilon,\eta}}{dt})\leq \vert\fracj{du^{\varepsilon,\eta}}{dt}\vert_{\a R^p}.\eeq
\\
Assuming that $\beta > 0$, it is clear that the solution $u^{\varepsilon,\eta}$ of (\ref{eq300}) is uniformly bounded versus $\eta$ in the space ${\Big [}H^1_0(]0,T[){\Big]}^p$. It is also obvious from (\ref{eq300}), that $\lambda^{\varepsilon,\eta}$ is uniformly bounded  versus $\eta$ in the space ${\Big [}L^\infty(]0,T[){\Big ]}^p$.
\\
Hence one can extract a subsequence from $(u^{\varepsilon,\eta},\lambda^{\varepsilon,\eta})$ denoted by $(u^{\varepsilon,\eta'},\lambda^{\varepsilon,\eta'})$ such that (see for instance H. Brezis \cite{BREZIS}):
\\
\beq\label{eq220}\left \{\begin{array}{l}
\limj_{\eta'\rightarrow 0}u^{\varepsilon,\eta'}=u^{\varepsilon,*}\hbox{ in the space } {\big [}H^1_0(]0,T[){\Big ]}^p-weakly,\\\\
\limj_{\eta'\rightarrow 0}\lambda^{\varepsilon,\eta'}= \lambda^{\varepsilon,*}\hbox{ in the space }{\big [}L^\infty(]0,T[){\Big ]}^p-weakly\;*.\end{array}\right.\eeq
\\
From the equation (\ref{eq202}) we deduce that:
\beq\label{eq230}\left \{\begin{array}{l}
\forall v\in  {\big [}H^1_0(]0,T[){\Big ]}^p,\\\\\varepsilon \intj_0^T\alpha (u^{\varepsilon,*},v)+\beta(\fracj{du^{\varepsilon,*}}{dt},\fracj{dv}{dt})+\gamma(\lambda^{\varepsilon,*},\fracj{dv}{dt})+\intj_0^T\;(^{t}\hskip-0.05cm{\mathcal B}P^{\varepsilon},v)=0.\end{array}\right.
\eeq
From (\ref{eq310}) one gets:
\beq\label{eq245}
\limj_{\eta' \rightarrow 0}(\lambda^{\varepsilon,\eta'},\fracj{du^{\varepsilon,\eta'}}{dt})=\vert\fracj{du^{\varepsilon,*}}{dt}\vert_{\a R^p}\hbox{  in the space }L^2(]0,T[).
\eeq
First of all let us check that $u^{\varepsilon,*}=u^\varepsilon$. From the equation (\ref{eq202})  in which we set $v=u^{\varepsilon,\eta'}$,  and because of (\ref{eq245}), we deduce that:
\beq\label{eq241}
\limj_{\eta'\rightarrow 0 }\intj_0^T{\Big [}\alpha \vert u^{\varepsilon,\eta'}\vert_{\a R^p}^2+\beta\vert \fracj{du^{\varepsilon,\eta'}}{dt}\vert_{\a R^p}^2{\Big ]}+\gamma \vert\fracj{du^{\varepsilon,*}}{dt}\vert_{\a R^p}+\fracj{1}{\varepsilon}\intj_0^T\;(^{t}\hskip-0.05cm{\mathcal B}P^\varepsilon,u^{\varepsilon,*})=0.
\eeq
Let us make use of the lower-semi-continuity for the weak topology of convex functions \cite{Baushke} and of the relation (\ref{eq210}). This leads to:
\beq\label{eq242}\intj_0^T{\Big [}\alpha \vert u^{\varepsilon,*}\vert_{\a R^p}^2+\beta\vert \fracj{du^{\varepsilon,*}}{dt}\vert^2{\Big ]}+\gamma\vert \fracj{du^{\varepsilon,*}}{dt}\vert_{\a R^p}+\fracj{1}{\varepsilon}\intj_0^T(\;^{t}\hskip-0.05cm{\mathcal B}P^\varepsilon,u^{\varepsilon,*})\leq 0,\eeq
and from (\ref{eq230}) we have also:
\beq\label{eq400}\forall v\in {\Big [}H^1_0(]0,T[){\Big ]},\;\intj_0^T\alpha (u^{\varepsilon,*},v)+\beta(\fracj{du^{\varepsilon,*}}{dt},\fracj{dv}{dt})+\gamma \vert\fracj{dv}{dt}\vert_{\a R^p}+\fracj{1}{\varepsilon}(\;^{t}\hskip-0.05cm{\mathcal B} P^{\varepsilon},v)\geq 0.\eeq
which implies by combining the two previous relations:
\beq\label{eq405}\hskip-0.6cm\left\{\begin{array}{l}\forall v\in {\Big [}H^1_0(]0,T[){\Big ]},\\\\
\;\hskip-0.25cm\intj_0^T\hskip-0.2cm\alpha (u^{\varepsilon,*},v\hskip-0.1cm-\hskip-0.1cmu^{\varepsilon,*})\hskip-0.1cm+\hskip-0.1cm\beta(\fracj{du^{\varepsilon,*}}{dt},\fracj{dv}{dt}\hskip-0.1cm-\hskip-0.1cm\fracj{du^{\varepsilon,*}}{dt})\hskip-0.1cm+\hskip-0.1cm\gamma (\vert\fracj{dv}{dt}\vert_{\a R^p}\hskip-0.1cm-\hskip-0.1cm\vert\fracj{du^{\varepsilon,*}}{dt}\vert_{\a R^p})\\\\+\fracj{1}{\varepsilon}\hskip-0.1cm(\;^{t}\hskip-0.05cm{\mathcal B} P^{\varepsilon},v\hskip-0.1cm-\hskip-0.1cmu^{\varepsilon,*})\hskip-0.1cm\geq \hskip-0.1cm0.\end{array}\right.\eeq
Therefore $u^{\varepsilon,*}$ is solution of the same inequation that $u^\varepsilon$. And from the uniqueness of this solution: $u^{\varepsilon,*}=u^{\varepsilon}$. Hence $u^\varepsilon$ is the only accumulation point of the sequence $u^{\varepsilon,\eta}$ (versus $\eta$) and therefore all the sequence $u^{\varepsilon,\eta}$ converges weakly to $u^\varepsilon$ when $\eta\rightarrow 0$ (otherwise one could extract another subsequence which would converge to another solution of the inequation (\ref{eq190})). 
\\\\
Let us now prove that  $u^{\varepsilon,\eta}$ converges strongly to $u^{\varepsilon}$ in the space ${\Big [}H^1_0(]0,T[){\Big ]}^p$. This will also ensure that $\vert\fracj{du^{\varepsilon,*}}{dt}\vert_{\a R^p}=(\lambda^{\varepsilon,*},\fracj{du^{\varepsilon,*}}{dt}).$ In other words, this will imply that almost everywhere if $\fracj{du^{\varepsilon,*}}{dt}\neq 0$ then $\lambda^{\varepsilon,*}=\pm 1$.\\\\
From the equations satisfied by $u^{\varepsilon,\eta}$ and $u^{\varepsilon}$ one gets ($u^\varepsilon=u^{\varepsilon,*}$):
\beq\label{eq250}\left \{\begin{array}{l}\forall v\in {\Big [}H^1_0(]0,T[){\Big ]}^p\\\\
\intj_0^T\alpha (u^{\varepsilon,\eta}-u^{\varepsilon},v)+\beta(\fracj{du^{\varepsilon,\eta}}{dt}-\fracj{du^{\varepsilon}}{dt},\fracj{dv}{dt})+\gamma(\lambda^{\varepsilon,\eta}-\lambda^{\varepsilon,*},\fracj{dv}{dt})=0.\end{array}\right.
\eeq
Setting $v=u^{\varepsilon,\eta}-u^{\varepsilon}$ in (\ref{eq250}), we deduce that:
\beq\label{eq270}\begin{array}{l}\hskip-0.05cm\limj_{\eta\rightarrow 0}\intj_0^T\hskip-0.05cm\alpha\vert u^{\varepsilon,\eta}\hskip-0.05cm-\hskip-0.05cmu^{\varepsilon}\vert_{\a R^p}^2\hskip-0.05cm+\hskip-0.05cm\beta \vert\fracj{du^{\varepsilon,\eta}}{dt}\hskip-0.05cm-\hskip-0.05cm\fracj{du^\varepsilon}{dt}\vert_{\a R^p}^2\hskip-0.05cm\\\\+\hskip-0.05cm\gamma\limj_{\eta\rightarrow 0}\intj_0^T\hskip-0.05cm(\lambda^{\varepsilon,\eta}\hskip-0.05cm-\hskip-0.05cm\lambda^{\varepsilon,*},\fracj{d(u^{\varepsilon,\eta}\hskip-0.05cm-\hskip-0.05cmu^{\varepsilon})}{dt})\hskip-0.05cm =\hskip-0.05cm0.\end{array}\eeq
But from the weak convergence and (\ref{eq245}):
$$\begin{array}{l}\limj_{\eta\rightarrow 0}\intj_0^T\hskip-0.05cm(\lambda^{\varepsilon,\eta}\hskip-0.05cm-\hskip-0.05cm\lambda^{\varepsilon,*},\fracj{d(u^{\varepsilon,\eta}\hskip-0.05cm-\hskip-0.05cmu^{\varepsilon})}{dt})\hskip-0.05cm =\\\\\limj_{\eta\rightarrow 0}\intj_0^T(\lambda^{\varepsilon,\eta},\fracj{du^{\varepsilon,\eta}}{dt})-\intj_0^T(\lambda^{\varepsilon,*},\fracj{du^{\varepsilon,*}}{dt})\geq\limj_{\eta\rightarrow 0}\intj_0^T(\lambda^{\varepsilon,\eta},\fracj{du^{\varepsilon,\eta}}{dt})-\intj_0^T\vert\fracj{du^{\varepsilon,\eta}}{dt}\vert_{\a R^p}=0,\end{array}$$
which implies that:
\beq\label{eq260}\hskip-.5cm\left \{\begin{array}{l}\limj_{\eta\rightarrow 0}\intj_0^T\alpha  \vert u^{\varepsilon,\eta}-u^\varepsilon\vert_{\a R^p}^2+\beta\vert \fracj{du^{\varepsilon,\eta}}{dt}-\fracj{du^\varepsilon}{dt}\vert_{\a R^p}^2=0,\\\\
\hbox{and by expliciting the second term and introducing the known limits: }\\\\
\gamma\intj_0^T\vert\fracj{du^\varepsilon}{dt}\vert_{\a R^p}={\limj_{\eta\rightarrow 0}\gamma \intj_0^T(\lambda^{\varepsilon,\eta},\fracj{du^{\varepsilon,\eta}}{dt})\leq \gamma \intj_0^T(\lambda^{\varepsilon,*},\fracj{du^{\varepsilon}}{dt}})\leq\gamma \intj_0^T\vert\fracj{du^\varepsilon}{dt}\vert_{\a R^p}\end{array}\right.\eeq
\\
which completes the proof of the strong convergence of $u^{\varepsilon,\eta}$ to $u^\varepsilon$ and also the relation which was not obvious (because (\ref{eq245}) is a weaker result):
 $$\intj_0^T(\lambda^{\varepsilon,*},\fracj{du^\varepsilon}{dt})=\intj_0^T\vert\fracj{du^\varepsilon}{dt}\vert_{\a R^p}.$$ 
In conclusion, using the trick of R. Glowinski \cite{GLOW} adapted to our case, we proved in this subsection, that the solution $u^\varepsilon\in {\Big [}H^1_0(]0,T[){\Big ]}^p$ can be obtained as a function of $P^\varepsilon$ by solving the following mixed formulation ($B_1$ is defined at (\ref{eq311}) and we have omitted the superscript $*$ in  the expression of $\lambda^\varepsilon$ as far as there is no ambiguity):
\beq\label{eq265}\left \{\begin{array}{l}\forall v\in {\Big [}H^1_0(]0,T[){\Big ]}^p:\\\\\varepsilon \intj_0^T{\Big [}\alpha (u^\varepsilon,v)+\beta(\fracj{du^\varepsilon}{dt},\fracj{dv}{dt})+\gamma(\lambda^\varepsilon,\fracj{dv}{dt}){\Big ]}+\intj_0^T(\;^{t}\hskip-.05cm{\mathcal B}P^\varepsilon,v)=0,\
\\\\\
\hbox{ for almost every }t\in ]0,T[,\;\forall \mu \in B_1,\;(\mu-\lambda^\varepsilon,\fracj{du^\varepsilon}{dt})\leq 0,\;\;\lambda^\varepsilon \in B_1.
\end{array}\right.\eeq
\subsubsection{Uniqueness of $(u^\varepsilon,\lambda^\varepsilon)$} Concerning $u^\varepsilon$ one could argue that the uniqueness is a consequence of the initial formulation of the model through the variational inequality. But let us give the result by another computation which also leads to the uniqueness of $\lambda^\varepsilon$ which has not yet been proved. Let us assume that there are two solutions denoted by $(u^{\varepsilon,i},\lambda^{\varepsilon,i}),\;i=1,2$  to the model explicited at (\ref{eq265}). Starting from the difference between the two relations obtained from (\ref{eq265}), we obtain:
\beq\label{eq266}\left \{\begin{array}{l}
\forall v\in {\Big [}H^1_0(]0,T[){\Big ]}^p:\\\\
\intj_0^T\alpha (u^{\varepsilon,1}-u^{\varepsilon,2},v)+\beta(\fracj{du^{\varepsilon,1}}{dt}-\fracj{du^{\varepsilon,2}}{dt},\fracj{dv}{dt})+\gamma(\lambda^{\varepsilon,1}-\lambda^{\varepsilon,2},\fracj{dv}{dt})=0,\\\\
\intj_0^T(\lambda^{\varepsilon,1}-\lambda^{\varepsilon,2}, \fracj{du^{\varepsilon,1}}{dt}-\fracj{du^{\varepsilon,2}}{dt})\geq 0.
\end{array}\right.\eeq
Setting $v=u^{\varepsilon,1}-u^{\varepsilon,2}$ in (\ref{eq266}) leads to the following informations:
\beq\label{eq267}\left \{ \begin{array}{l}
\intj_0^T\alpha  \vert u^{\varepsilon,1}-u^{\varepsilon,2}\vert_{\a R^p}^2+\beta 
\vert\fracj{du^{\varepsilon,1}}{dt}-\fracj{du^{\varepsilon,2}}{dt}\vert_{\a R^p}^2=0\;=>\;u^{\varepsilon,1}=u^{\varepsilon,2},\\\\
\intj_0^T(\lambda^{\varepsilon,1}-\lambda^{\varepsilon,2},\fracj{dv}{dt})=0\;=>\;\fracj{d\lambda^{\varepsilon,1}}{dt}=\fracj{d\lambda^{\varepsilon,2}}{dt}\;=>\; \lambda^{\varepsilon,1}=\lambda^{\varepsilon,2}+c.
\end{array}\right.\eeq
We proved (for any solution $\lambda^\varepsilon$) that for almost any $t\in ]0,T[$ one has $\lambda^{\varepsilon}=\pm 1$, the sign depending on the one of ${du^\varepsilon}/{dt}$. Hence if $u^\varepsilon\neq 0$  on a segment of $]0,T[$ ($u^\varepsilon=0$ is the only case where one can have ${du^\varepsilon}/{dt}=0$ on a segment  of $]0,T[$) the constant $c$ must be zero. Let us notice that if we would have $u^\varepsilon=0$, we deduce  from (\ref{eq265}) that: $\;^{t}\hskip-0.05cm {\mathcal B}P^\varepsilon=0$. From the controllability property this requires that $P^\varepsilon=0$ which means that $u^\varepsilon=0$ is an exact control and the problem is solved without any computation.  This proves the uniqueness of $\lambda^\varepsilon$ unless if $u^\varepsilon=0$ is the solution of the optimal control problem.
  \subsection{{\it Pseudo} asymptotic analysis versus the marginal cost $\varepsilon$ of the control}
  The choice of the small parameter $\varepsilon$ is not obvious. Anyway the formulation of the optimal control strategy is meaningful if and only if a controllability property is satisfied by the system. This exact controllability property for ordinary differential equations, was first stated independently and almost simultaneously by R. Bellman and L. Pontryagin \cite{BEL}-\cite{PONT}. 
  \begin{definition}\label{def1}The system (\ref{eq150}) is controllable at time $T$ if for any initial condition of the state equation $X_0\in \a R^N$, $X_1\in \a R^N$ and any right-hand side $F\in {\Big [}L^2(]0,T[){\Big ]}^N$, there exists a control $u\in { [}H^1_0(]0,T[){ ]}^p$ such that: $X(T)=\dot X(T)=0$. In fact the time $T>0$ is arbitrary as far as there is no bound on the control. \hfill $\Box$
  \end{definition}
 \noindent Let us give another definition which is adapted to our problem and which is equivalent in our case, to the one of R. Bellman and L. Pontryagin.
  \begin{definition}\label{def2} Let $Q$ be a solution (the initial conditions are not prescribed) of the equation:
  \beq\label{eq500} {\mathcal M}\ddot Q-\;^{t}\hskip-0.05cm{\mathcal C}\dot Q-\;^{t}\hskip-0.05cm{\mathcal K}Q=0\hbox{ and }^{t}\hskip-0.05cm{\mathcal B}Q=0.\eeq
  The system is exactly controllable if and only if $Q\equiv 0$. \hfill $\Box$\end{definition}
  %
 \noindent We assume that the exact controllability property (see Definitions  \ref{def1}-\ref{def2}) is satisfied. Our goal is to prove that the optimal control $u^\varepsilon$ solution of the model (\ref{eq170}) has a limit when $\varepsilon \rightarrow 0$ and to suggest a way to compute it.
  \\
  Because the optimality relations are not linear (as far as $\gamma\neq 0$), a special analysis is required. 
  Let us set  {\it a priori} (but it is clear at this step that this is not justified):
  \beq\label{eq600}\left \{\begin{array}{l}
  u^\varepsilon=u^0+\varepsilon u^1+ \ldots, 
  \\
  X^\varepsilon=X^0+\varepsilon X^1+ \ldots,
   \\
  P^\varepsilon=P^0+\varepsilon P^1+ \ldots
  \\
  \lambda^\varepsilon=\lambda^0+\ldots
  \end{array}\right.
  \eeq
Nothing guarantees that this assumed asymptotic expansion makes sense. Nevertheless we use it as a guide for our purpose. Let us point out several necessary relations:  
\beq\label{eq610}\hskip-.5cm\left\{\begin{array}{l}{\mathcal M}\ddot X^0+{\mathcal C}\dot X^0+{\mathcal K}X^0={\mathcal B}u^0+{\mathcal F},\;X^0(0)=X_0,\;\dot X^0(0)=X_1,
\\ \\
\forall v\in {\Big [}H^1_0(]0,T[){\Big ]}^p,\;\intj_0^T(\;^{t}\hskip-0.05cm{\mathcal B}P^0,v)=0=> ^{t}\hskip-0.05cm{\mathcal B}P^0=0,\\\\{\mathcal M}\ddot P^0-\;^{t}{\mathcal C}\dot P^0+\;^{t}\hskip-0.05cm{\mathcal K}P^0=0,\\\\ {\mathcal M}P^0(T)=\dot X(T),\;{\mathcal M}\dot P^0(T)=-X^0(T)+\;\hskip-0.05cm ^t{\mathcal C}P^0(T),\\\\\
{\mathcal M}\ddot P^1-\;^{t}{\mathcal C}\dot P^1+\;^{t}{\mathcal K}P^1=0,\;P^1(0)=\Phi_0\in \a R^N,\;\dot P^1(T)=\Phi_1\in \a R^N,\\\\
\forall v\in {\Big [}H^1_0(]0,T[){\Big ]}^p,\\\\\intj_0^T{\Big [}\alpha (u^0,v)+\beta (\fracj{du^0}{dt},\fracj{dv}{dt})+\gamma (\lambda^0,\fracj{dv}{dt}){\Big ]}=-\intj_0^T\;(^{t}\hskip-0.05cm{\mathcal B}P^1,v),\\\\
\hbox{ for almost any }t\in ]0,T[,\;\lambda^0\in B _1,\;\;\forall \mu\in B_1,\;(\mu-\lambda^0,\fracj{du^0}{dt})\leq 0.
\end{array}\right.\eeq
Because of the non-linearity of the last inequation we can only state a necessary condition for the term of order zero in $\varepsilon$, even if we underline that, at this step, we do not know if this is correct or not.\\
Let us recall that one can eliminate $\lambda^0$ as we did for $u^\varepsilon$,  in order to obtain a characterization of $u^0$ by a variational inequation:
\beq\label{eq611}\hskip-0.5cm\left \{\begin{array}{l}
\forall v\in {\Big [}H^1_0(]0,T[){\Big ]}^p,
\;\;\hskip-0.1cm\intj_0^T\hskip-0.1cm\alpha (u^0,v-u^0)\hskip-0.05cm+\hskip-0.05cm\beta(\fracj{du^0}{dt},\fracj{dv}{dt}\hskip-0.05cm-\hskip-0.05cm\fracj{du^0}{dt})\hskip-0.05cm\hskip-0.05cm\\\\ +\gamma\intj_0^T(\vert\fracj{dv}{dt}\vert_{\a R^p}\hskip-0.05cm-\hskip-0.05cm\vert \fracj{du^0}{dt}\vert_{\a R^p})\hskip-0.05cm+\intj_0^T\hskip-0.05cm\;(^{t}\hskip-0.05cm{\mathcal B}P^1,v-u^0)\hskip-0.05cm\geq \hskip-0.05cm0.\end{array}\right.
\eeq
But the existence of a solution to the full system (including $u^0,\lambda^0,X^0$ and $P^1$) is not yet proved.
The goal of the rest of this subsection is to prove that (\ref{eq610}) enables to characterize this term uniquely and to give an algorithm for computing it. We also prove in the following that when $\varepsilon \rightarrow 0$ the following result holds (we recall that it  has been proved that $P^0$ is necessarily equal to zero because of the controllability property) :  
$$(u^\varepsilon,\lambda^\varepsilon,X^\varepsilon,\fracj{P^\varepsilon}{\varepsilon})  \hbox{ converges to }(u^0,\lambda^0,X^0,P^1)$$
 in the space:
 $${\Big [}H^1_0(]0,T[){\Big ]}^p\times {\Big [}L^\infty(]0,T[){\Big ]}^p\times {\Big [}{\mathcal C}^1([0,T]){\big ]}^N\times {\Big [}{\mathcal C}^1([0,T]){\big ]}^N.$$
  \subsubsection{A few notations} Let us set $\Psi=(\Psi_0,\Psi_1)\in \a R^{2N}$ and we associate $Q(\Psi)$ solution of:
  \beq\label{eq620}{\mathcal M}\ddot Q-\;^{t}\hskip-.05cm {\mathcal C}\dot Q+\;^{t}\hskip-.05cm{\mathcal K}Q=0,\;Q(0)=\Psi_0,\;\dot Q(0)=\Psi_1.\eeq\\
  Let $(u(\Psi),\lambda(\Psi))\in {\Big [}H^1_0(]0,T[){\Big ]}^p\times  {\Big [}L^\infty(]0,T[){\Big ]}^p$ be the solution of:
  \beq\label{eq630}\hskip-.4cm\left \{\begin{array}{l}
  \forall v\in{\Big [}H^1_0(]0,T[){\Big ]}^p,\\\\\intj_0^T\hskip-0.15cm{\Big [}\alpha (u(\Psi),v)\hskip-0.05cm+\hskip-0.05cm\beta(\fracj{du(\Psi)}{dt},\fracj{dv}{dt})\hskip-0.05cm+\hskip-0.05cm\gamma(\lambda(\Psi),\fracj{dv}{dt}){\Big ]}\hskip-0.05cm=\hskip-0.1cm-\hskip-0.1cm\intj_0^T\hskip-0.15cm(\;^{t}\hskip-0.05cm{\mathcal B}Q(\Psi)v),\\\\
  \hbox{where for almost any } t\in ]0,T[,\;\lambda(\psi)\in B_1\hbox{ and }\forall \mu \in B_1,\;(\mu-\lambda(\Psi),\fracj{du(\Psi)}{dt})\leq 0,\\\\(\hbox{hence }(\lambda(\Psi),\fracj{du(\Psi)}{dt})=\vert\fracj{du(\Psi)}{dt}\vert_{\a R^p}) \\\\ \hbox{ more precisely  }\lambda =\hbox{sign}(\fracj{du(\psi)}{dt}) \hbox{ almost everywhere }\fracj{du(\psi)}{dt}\neq 0.
  \end{array}\right.\eeq
  For any $\Psi\in \a R^{2N}$ we introduce the non linear operator $\Lambda$ from $\a R^{2N}$ into itself and defined by:
  \beq\hskip-0.05cm\forall \delta\Psi\in \a R^{2N},\;(\Lambda(\Psi),\delta\Psi)\hskip-0.05cm=\hskip-0.05cm-\intj_0^T\hskip-0.05cm(\;^{t}\hskip-0.05cm{\mathcal B}Q(\delta \Psi),u(\Psi)),\; (\hbox{let us notice that: }\Lambda(0)\hskip-0.05cm=\hskip-0.05cm0).\eeq
  Let us now give several properties of the operator $\Lambda$. 
  \subsubsection{Monotonicity of $\Lambda$}
  First of all there is a strictly positive constant $c_0$ independent on $\Psi$ and such that:
 \beq\label{eq640}\begin{array}{l}(\Lambda(\Psi),\Psi)=-\intj_0^T(\;^{t}\hskip-0.05cm{\mathcal B}Q(\Psi),u(\psi))=
 \\\\
 \intj_0^T\hskip-0.1cm{\Big [}\alpha \vert u(\Psi)\vert_{\a R^p}^2\hskip-0.05cm+\hskip-0.05cm\beta\vert\fracj{du(\Psi)}{dt}\vert_{\a R^p}^2{\Big ]}\hskip-0.05cm+\gamma\vert\fracj{du(\Psi)}{dt}\vert_{\a R^p}\geq\hskip-0.05cm c_0\vert\vert u(\Psi)\vert\vert^2_{1,]0,T[}\end{array}\eeq
 Therefore, if $\Lambda(\Psi)=0$, then $u(\Psi)=0$. 
Let us now consider two vectors of $\a R^{2N}$ -say $\Psi^1$ and $\Psi^2$. One has ($Q$ is linear versus $\Psi$ but not $u$):
$$\begin{array}{l}(\Lambda(\Psi^1)-\Lambda(\Psi^2),\Psi^1-\Psi^2)=-\intj_0^T(\;^{t}\hskip-0.05cm{\mathcal B}Q(\Psi^1-\Psi^2),u(\Psi^1)-u(\Psi^2))=\\\\
\hskip-0.1cm\intj_0^T\hskip-0.1cm\alpha \vert u(\Psi^1)\hskip-0.05cm-\hskip-0.05cmu(\Psi^2)\vert^2\hskip-0.05cm+\hskip-0.05cm\beta\vert \fracj{du(\Psi^1)}{dt}\hskip-0.05cm-\hskip-0.05cm\fracj{du(\Psi^2)}{dt}\vert^2\hskip-0.05cm+\hskip-0.05cm\gamma(\lambda(\Psi^1)\hskip-0.05cm-\hskip-0.05cm\lambda(\Psi^2),\fracj{du(\Psi^1)}{dt}\hskip-0.05cm-\hskip-0.05cm\fracj{du(\Psi^2)}{dt}).
\end{array}
$$
$$\begin{array}{l}\hbox{From }(\lambda(\Psi)\in B_1):\;\;(\lambda(\Psi^1)\hskip-0.05cm-\hskip-0.05cm\lambda(\Psi^2),\fracj{du(\Psi^1)}{dt}\hskip-0.05cm-\hskip-0.05cm\fracj{du(\Psi^2)}{dt})=\\\\\ \vert\fracj{du(\Psi^1)}{dt}\vert_{\a R^p}+\vert\fracj{du(\Psi^2)}{dt}\vert_{\a R^p}-(\lambda(\Psi^1),\fracj{du(\Psi^2)}{dt})-(\lambda(\Psi^2),\fracj{du(\Psi^1)}{dt})\geq 0,\end{array}$$
we deduce that:
\beq\label{eq660}(\Lambda(\Psi^1)-\Lambda(\Psi^2),\Psi^1-\Psi^2)\geq\\
\hskip-0.1cm\intj_0^T\hskip-0.1cm\alpha \vert  u(\Psi^1)\hskip-0.05cm-\hskip-0.05cmu(\Psi^2)\vert_{\a R^p}^2\hskip-0.05cm+\hskip-0.05cm\beta\vert \fracj{du(\Psi^1)}{dt}\hskip-0.05cm-\hskip-0.05cm\fracj{du(\Psi^2)}{dt}\vert_{\a R^p}^2,\eeq
hence  there exists a strictly positive constant $c_2$ (we can choose $c_2=c_0!$) independent neither on $\Psi^1$ nor on $\Psi^2$ such that:
\beq\label{eq661}(\Lambda(\Psi^1)-\Lambda(\Psi^2),\Psi^1-\Psi^2)\geq c_2\vert\vert u(\Psi^1)-u(\Psi^2)\vert\vert_{1,]0,T[}^2.\eeq
Therefore the operator $\Lambda$ is strictly monotone. But one can say more.
\\
\subsubsection{Coerciveness of  $\Lambda$} Let us notice that:
\beq\left \vert\label{eq670}\begin{array}{l}\forall v\in{\Big [}H^1_0(]0,T[){\Big ]}^p \hbox{ one has from the definition of } u(\Psi):
\\
-\intj_0^T (\;^{t}\hskip-0.05cm{\mathcal B}Q(\Psi),v)=\intj_0^T\alpha (u(\Psi),v)+\beta (\fracj{du(\Psi)}{dt},\fracj{dv}{dt})+\gamma(\lambda(\Psi),\fracj{dv}{dt}),
\\
\hbox{ and }\hbox{for almost any }t\in ]0,T[:
\\
\lambda(\Psi)\in B_1\hbox{ and }\forall \mu\in B_1,\;(\mu,\fracj{du(\Psi)}{dt})\leq \vert\fracj{du(\Psi)}{dt}\vert_{\a R^p}.
\end{array}\right.\eeq
Let $\delta\in \a R^{+}$ such that $0\leq \delta \leq 1$. If we denote by 
$u(\delta\Psi)$ the solution of (\ref{eq670}) with $Q(\delta \Psi)$ given, one can check directly that $u(\delta\Psi)=\delta u(\Psi)$ and that $\delta\lambda(\Psi)\in B_1$ is a solution (not uniquely defined!) associated to $u(\delta\Psi)$,  is a solution of (\ref{eq670}) (where $\Psi$ is replaced by $\delta\Psi$). We restrict $\delta$ to the segment $[0,1]$ in order to have $\delta\lambda(\Psi)\in B_1$. Let us also notice that because of the exact controllability assumption (see Definition \ref{def2}), the quantity:
$$\Psi\in \a R^{2N}\rightarrow \vert\vert^{t}{\mathcal B}P(\Psi)\vert\vert_{-1,]0,T[},$$
is a norm on $\a R^{2N}$ equivalent to any other norm on $\a R^{2N}$ (because it is a finite dimensional space).
\\
Let us turn to a first result which will be used for the {\it  coerciveness} of $\Lambda$.
\begin{theorem}\label{th10} Let $d>0$ be a constant sufficiently large (this is explicited in the proof). We define  $D^d$ the disc of $\a R^{2N}$ by:
$$D^d=\{\Psi \in \a R^{2N},\; \vert\vert^{t}{\mathcal B}P(\Psi)\vert\vert_{-1,]0,T[},\leq d\}$$
and its boundary $\partial D^d$ by :
$$\partial D^d=\{\Psi\in \a R^{2N},\;\vert\vert^{t}{\mathcal B}P(\Psi)\vert\vert_{-1,]0,T[},=d\}.$$
Then:
$$\exists c_3>0 \hbox{ such that }\forall \Psi\in \partial D^d,\;c_3\vert\vert^{t}{\mathcal B}P(\Psi)\vert\vert_{-1,]0,T[}\leq \vert\vert u(\Psi)\vert\vert_{1,]0,T[}.$$
($c_3$ is obviously independent on $\Psi$ but can depend on $d$).
\hfill ${\Box}$\end{theorem}
\begin{proof} Let us use an absurdist reasoning. If Theorem \ref{th10} is false, then for any $n>0$ there exists an element $\Psi^n\in \partial D^d$ such that:
$$\vert\vert u(\Psi^n)\vert\vert_{1,]0,T[}\leq \fracj{d}{n}.$$
Hence, because $\partial D^d$ is bounded in $\a R^{2N}$,  there is a subsequence -say $\Psi^{n'}$- of $\Psi^n$ which converges to an element $\Psi^*\in \partial D^d$ and the corresponding solution $u(\Psi^{n'})$ converges to $0$ in the space ${\Big [}H^1_0(]0,T[){\Big ]}^p$. From the variational characterization (the inequation) of $u(\Psi^{n'})$ one has:
$$\begin{array}{l}\intj_0^T\alpha(u(\Psi^{n'}),u(\Psi^{n'}))+\beta(\fracj{du(\Psi^{n'})}{dt},\fracj{du(\Psi^{n'})}{dt})+\gamma\vert\fracj{du(\Psi(^{n'})}{dt}\vert\\\\-\intj_0^T(^{t}{\mathcal B}P(\Psi^{n'}),v-u(\Psi^{n'}))\leq 
\intj_0^T\alpha(u(\Psi^{n'}),v)+\beta(\fracj{du(\Psi^{n'})}{dt},\fracj{dv}{dt})+\gamma\vert\fracj{dv}{dt}\vert.\end{array}$$
And at the limit when $n'\rightarrow 0$:
$$\forall v\in {\Big [}H^1_0(]0,T[){\Big ]}^p,\;\fracj{-\intj_0^T(^{t}{\mathcal B}P(\Phi^*),v)}{\vert\vert v\vert \vert_{1,]0,T[} }\leq \gamma \sqrt{T}\vert\vert v\vert\vert_{1,]0,T[},$$
which implies that:
$$d=\vert\vert ^{t}{\mathcal B}P(\Psi^*)\vert\vert_{ -1,]0,T[} \leq\gamma\sqrt{T}.$$
Finally if $d>\gamma \sqrt{T}$ we obtain the contradiction and this proves Theorem \ref{th10}.
\end{proof}
\noindent We can prove now the second useful result for the {\it coerciveness} of $\Lambda$.
\begin{theorem}\label{th12} Assuming that $d>\gamma \sqrt{T}$ there is a strictly positive constant $c_4$ (one can choose $c_4=c_3!$) such that one has the following inequality (with the same notations as in Theorem \ref{th10}):
$$\forall \Psi\in D^d,\;c_4\vert\vert^{t}{\mathcal B}P(\Psi)\vert\vert_{-1,]0,T[}\leq \vert\vert u(\Psi)\vert\vert_{1,]0,T[}.$$
\hfill $\Box$\end{theorem}
\begin{proof} Let us consider an arbitrary element $\Psi\in  D^d$ and let us introduce a coefficient $\delta\in [0,1]$ such that: $\Psi=\delta \tilde \Psi$ where $\tilde\Psi\in \partial D^d$. If $\delta=0$ the Theorem \ref{th12} is obvious. If $\delta >0$ one has from the previous considerations: $u(\Psi)= u(\delta\tilde \Psi)=\delta u(\tilde \Psi)$. And applying Theorem \ref{th10} to $\tilde \Psi$ (at this step we can say that $c_4=c_3$):
$$c_3\vert\vert ^{t}{\mathcal B}P(\tilde \Psi)\vert\vert_{-1,]0,T[}\leq \vert\vert u(\tilde \Psi)\vert\vert_{1,]0,T[},$$
or else by multiplying by $\delta$:
$$ c_3\vert\vert ^{t}{\mathcal B}P(\Psi)\vert\vert_{-1,]0,T[}\leq \vert\vert u(\Psi)\vert\vert_{1,]0,T[}.$$
This completes the proof of Theorem \ref{th12}. This result proves also that:
$$\hbox{ if }\vert\vert ^{t}{B}P(\Psi)\vert\vert_{_1,]0,T[}\rightarrow \infty \hbox{ then }\vert\vert u(\Psi)\vert\vert_{1,]0,T[}\rightarrow \infty,$$
which can be interpreted as a {\it coerciveness} property of $u(\Psi)$ versus $\Psi$.
\end{proof}
\vskip .2cm
\noindent Let us turn now to a more precise characterization of the the strict monotonicity of $\Lambda$. 
\\\\
\noindent Let $\Psi^1$ and $\Psi^2$ be two elements of $ D^d$ and $u^1$ and $u^2$ the corresponding solutions of the variational inequations for $\Psi^1$ and $\Psi^2$ given. Let us use again an absurdist reasoning as we did in the proof of Theorem \ref{th10}. Thus, we assume in a first step that :
\beq\label{eq1000}\vert\vert ^{t}{\mathcal B}P(\Psi^1\hskip-.05cm-\hskip-.05cm\Psi^2)\vert\vert_{-1,]0,T[} \;(\simeq \hskip-.05cm\vert \Psi^1\hskip-.05cm-\hskip-.05cm\Psi^2\vert_{\a R^{2N}})=\hskip-.05cm g>0; (\hbox{ one has obviously }g\hskip-.05cm\leq \hskip-.05cm2d).\eeq
Let us assume that there is no strictly positive constant $c^*$ (independent neither on $\Psi^1$ nor on $\Psi^2$ such that:
\beq\label{eq1001}\forall \Psi^1,\Psi^2\in D^d, \hbox{satisfying (\ref{eq1000})},\;c^*\vert \Psi^1-\Psi^2\vert_{\a R^{2N}}\leq \vert\vert u^1-u^2\vert\vert_{1,]0,T[}.\eeq
Then, for any $n\in \a N^*$,\;there exist two sequences $\Psi^1_n$ and $\Psi^2_n$ in $D^d$, such that $$\vert \Psi^1-\Psi^2\vert_{\a R^{2N}}=g.$$ They are associated to the solutions $u^1_n$ and $u^2_n$ of the variational equation such that:
\beq\label{eq1002}
\vert \vert u^1_n-u^2_n\vert\vert_{1,]0,T[}\leq \fracj{1}{n},\;\vert \Psi^1_n-\Psi^2_n\vert_{\a R^{2N}}=g,\;\vert \Psi^1_n\vert_{\a R^{2N}}\leq d,\;\vert\Psi^2_n\vert_{\a R^{2N}}\leq d.
\eeq
Thus one can conclude that there are subsequences $\Psi^1_{n'},\Psi^2_{n'}$ and $u^1_{n'},u^2_{n'}$ which converge  to elements $\Psi^1_*,\Psi^2_*$ in $\a R^{2N}$ and to $u^1_*,u^2_*$ weakly in the space ${\Big [}H^1_0(]0,T[){\Big ]}^p$ and $u^1_*=u^2_*$ (strong convergence to zero of $u^1_n-u^2_n$ in the space ${\Big [}H^1_0(]0,T[){\Big]}^p$).
\\
Let us recall the inequations satisfied by $u^1$ and $u^2$:
\beq\label{eq1003}\left \{\begin{array}{l}
\forall\;i=1,2,\;\forall v\in {\Big [}H^1_0(]0,T[){\Big]}^p:\;\;
\intj_0^T\alpha(u^{i},v-u^{i})+\beta(\fracj{du^{i}}{dt},\fracj{dv}{dt}-\fracj{du^{i}}{dt})\\\\+\gamma\intj_0^T(\vert\fracj{dv}{dt}\vert-\vert\fracj{du^{i}}{dt}\vert)\geq -\intj_0^T\;^{t}{\mathcal B}P^{i}(v-u^{i}).
\end{array}\right.
\eeq
Let us choose $v=w+u^1$ for $i=1$ and $v=-w+u^2$ for $i=2$.  From:
$$\vert\fracj {d(\pm w +u^{i})}{dt}\vert\leq \vert\fracj {dw}{dt}\vert +\vert\fracj{d u^{i}}{dt}\vert,$$
this leads after convergence, to:
\beq\label{eq1004}\left \{\begin{array}{l}\forall w\in {\Big [}H^1_0(]0,T[){\Big]}^p:\;\;
2\gamma \intj_0^T\vert \fracj{dw}{dt}\vert\geq -\intj_0^T\;^{t}{\mathcal B}P(\Psi^1_*-\Psi^2_*).w
\end{array}\right.
\eeq
Or else, using Cauchy-Schwarz triangular inequality and from the definition of the norm in the functional space $H^{-1}(]0,T[)$:
\beq\label{eq2005}\vert\vert ^{t}{\mathcal B} P(\Psi^1_*-\Psi^2_*)\vert\vert_{-1,]0,T[}\leq 2\gamma\sqrt{T}.\eeq
Hence we get a contradiction if we choose $g>2\gamma\sqrt{T}$. Therefore the inequality (\ref{eq1001}) is true.
\\
Let us now consider another couple $\Psi^1,\Psi^2$ associated to the solution $u^1,u^2$ and such that: $$\vert\Psi^1-\Psi^2\vert\leq g.$$
If $\Psi^1=\Psi^2$ the inequality (\ref{eq1001}) is still true. \\If $\Psi^1\neq \Psi^2$ we set $\Psi^1=\kappa\tilde\Psi^1,\Psi^2=\kappa\tilde\Psi^2$ where $\kappa=\fracj{\vert\Psi^1-\Psi^2\vert_{\a R^{2N}}}{g}\leq 1$.\\
The inequality (\ref{eq1001}) is true for $\tilde \Psi^1$ and $\tilde \Psi^2$. But $\Psi^1=\kappa \tilde\Psi^1,\Psi^2=\kappa \tilde\Psi^2$ are associated to $u^1=\kappa \tilde u^1,u^2=\kappa\tilde u^2$ and therefore the inequality (\ref{eq1001}) is still true for $\Psi^1,\Psi^2$ and $u^1,u^2$. Finally we summarize the obtained result in the following statement.
\begin{theorem}\label{th11} The exact controllability is  assumed. Then there exists a constant $c^*>0$ such that for any $\Psi^1,\Psi^2 \in D^d$ and $u^1,u^2$ the solution associated, one has:
\beq \label{eq1006}c^*\vert\Psi^1-\Psi^2\vert_{\a R^{2N}}\leq \vert\vert u^1-u^2\vert\vert_{1,]0,T[}.\eeq
\;\hfill $\Box$
  \end{theorem}
\begin{remark}\label{rem20} The constant $c^*$ can depend on $d$. But $d$ is arbitrary as soon it is strictly larger than $\gamma \sqrt{T}$. Hence the result of Theorem \ref{th11} is always true.\hfill $\Box$\end{remark}
\vskip.2cm
\begin{remark}\label{rem21} If for instance we set $\Psi^2=0$ we get back to Theorem \ref{th10}. Therefore one could object that it was not necessary to introduce it. But we believe that it was more convenient for the reader to separate the two Theorems \ref{th10} and \ref{th11}. \hfill $\Box$\end{remark}
\vskip.2cm
\noindent The inequality of Theorem \ref{th11} proves the  {\it  coerciveness} of the strictly monotone operator $\Lambda$ on $\a R^{2N}$.
Let us turn now to the continuity. 
\subsubsection{Continuity of $\Lambda$} From:
\beq\label{eq710} \left \{\begin{array}{l}(\Lambda(\Psi^1)-\Lambda(\Psi^2),\Psi^1-\Psi^2)=-\intj_0^T(\;^{t}\hskip-0.05cm{\mathcal B}Q(\Psi^1-\Psi^2),u(\Psi^1)-u(\Psi^2))\leq
\\\\
\vert\vert\; ^{t}\hskip-0.05cm{\mathcal B}Q(\Psi^1-\Psi^2)\vert\vert_{-1,]0,T[}\;\vert\vert u(\Psi^1)-u(\Psi^2)\vert\vert_{1,]0,T[},
\end{array}\right.\eeq
\\
and from (\ref{eq661}) recalling that the norm $\vert\vert\; ^{t}\hskip-0.05cm{\mathcal B}Q(\Psi^1-\Psi^2)\vert\vert_{-1,]0,T[}$ is equivalent to any norm on $\a R^{2N}$ (finite dimensional space), there exists a strictly positive constant $c_5$ such that:
\beq\label{eq720}(\Lambda(\Psi^1)-\Lambda(\Psi^2),\Psi^1-\Psi^2)\leq c_5\vert \Psi^1-\Psi^2\vert_{\a R^{2N}}^2.\eeq
We can also write:
\beq\label{eq701}\begin{array}{l}\vert \Lambda(\Psi^1)-\Lambda(\Psi^2)\vert_{\a R^{2N}}=\sup_{\delta\Psi\in \a R^{2N}}\fracj{(\Lambda(\Psi^1)-\Lambda(\Psi^2),\delta\Psi)}{\vert \delta \Psi\vert_{\a R^{2N}}}=
\\\\
\sup_{\delta\Psi\in \a R^{2N}}\fracj{-\intj_0^T(\;^{t}\hskip-0.05cm{\mathcal B}Q(\delta\Psi),u(\Psi^1)-u(\Psi^2))}{\vert \delta \Psi\vert_{\a R^{2N}}}\leq \\\\\sup_{\delta\Psi\in \a R^{2N}}\fracj{\vert\vert \;^{t}\hskip -0.05cm{\mathcal B}Q(\delta\Psi)\vert\vert_{-1,]0,T[}} {\vert \delta \Psi\vert_{\a R^{2N}}} \;\vert\vert u(\Psi^1)-u(\Psi^2)\vert\vert_{1,]0,T[}\leq \\\\c_5\vert\vert u(\Psi^1)-u(\Psi^2)\vert\vert_{1,]0,T[}.
\end{array}
\eeq
From the inequality (\ref{eq661}), we deduce:
\\\\
$$\begin{array}{l}\sqrt{c_0}\vert\vert u(\Psi^1)-u(\psi^2)\vert\vert_{1,]0,T[}\leq \sqrt{(\Lambda(\Psi^1)-\Lambda(\Psi^2),\Psi^1-\Psi^2)}\leq \\\\\sqrt{\vert (\Lambda(\Psi^1)-\Lambda(\Psi^2))\vert_{\a R^{2N}}}\;\sqrt{\vert\Psi^1-\Psi^2\vert_{\a R^{2N}}}.\end{array}$$
Finally we obtain the estimate:
\beq\label{eq702}\vert \Lambda(\Psi^1)-\Lambda(\Psi^2)\vert_{\a R^{2N}}\leq (\fracj{c_5^2}{c_0})\vert \Psi^1-\Psi^2\vert_{\a R^{2N}}.\eeq
Let us discuss now the existence and uniqueness of a solution to the following equation for any arbitrary vector $L\in \a R^{2N}$:
\beq\label{eq730} \Lambda(\Phi)=L.
\eeq
%
\begin{theorem}\label{th15}
Let us assume  that the exact controllability property is satisfied. For any vector $L\in \a R^{2N}$ there is a unique solution $\Phi\in \a R^{2N}$ to equation (\ref{eq730}). Furthermore there exists a strictly positive constant $c_5$ such that:
\hfill $\Box$
\end{theorem}
%
\begin{proof} Let us define a sequence of $\a R^{2N}$ by the algorithm:
$$\Phi^{n+1}=\Phi^n-\varrho[\Lambda(\Phi^n)-L],\; \varrho\in \a R^{+*},\;\Phi^0\hbox{ being arbitrary, for instance }\Phi^0=0.$$
One has, using the two relations (\ref{eq1006})-(\ref{eq702}):
$$\hskip-0.1cm\begin{array}{l}\vert \Phi^{n+1}-\Phi^n\vert_{\a R^{2N}}^2=vert \Phi^{n}-\Phi^{n-1}\vert_{\a R^{2N}}^2\hskip-0.05cm-\hskip-0.05cm2\varrho (\Lambda(\Phi^n)-\Lambda(\Phi^{n-1}),\Phi^{n}-\Phi^{n-1})\hskip-0.05cm\\\\ +\hskip-0.05cm\varrho^2\vert\Lambda(\Phi^n)-\Lambda (\Phi^{n-1})\vert_{\a R^{2N}}^2
\leq \vert \Phi^{n}-\Phi^{n-1}\vert_{\a R^{2N}}^2(1-2\varrho c_*+(\fracj{c_5^4}{c_0^2})\varrho^2)\end{array}$$
Choosing $0<\varrho<2\fracj{c_0^2c_*}{c_5^4}$ we can conclude that  $\Phi^n$ is a Cauchy sequence one and therefore converges to an unique element $\Phi$ which is solution of (\ref{eq730}). The {\it a priori} estimate on $\Phi$ is a consequence of (\ref{eq1006}) because:
\beq\label{eq760}
\vert L\vert_{\a R^{2N}}=\sup_{\delta \Psi\in \a R^{2N}}\fracj{ (\Lambda (\Phi),\delta\Psi)}{\vert \delta \Psi\vert_{\a R^{2N}}}\geq
\fracj{ (\Lambda (\Phi),\Phi)}{\vert \Phi\vert_{\a R^{2N}}}\geq c_0\fracj{\vert\vert u(\Phi)\vert\vert^2_{1,]0,T[}}{\vert \Phi\vert_{\a R^{2N}}}\geq c_0c_*\vert \Phi\vert_{\a R^{2N}}.
\eeq
\end{proof} 
\begin{remark}\label{rem11} It is worth to notice for the numerical applications that all the constants which appear in the previous estimates, can be computed or at least estimated with a reduced computational cost. Nevertheless, an optimal step for $\varrho$ is certainly a good choice (many tricky suggestions for the choice of $\varrho$ can be found in  J. Cea \cite{CEA}).
\hfill $\Box$\end{remark}
%
\subsubsection{Global solution of the {\it pseudo} asymptotic model} Let us come back to the initial goal. We aim at finding a control $u$ such that $X(T)=\dot X(T)=0$ where $X$ is the solution of the state equation with the control $u$. For any $\delta \Phi=(\delta\Phi_0,\delta\Phi_1)\in \a R^{2N}$, we introduce $Q(\delta \Phi)$ the solution of:
\beq\label{eq800}{\mathcal M}\ddot Q-\;^{t}\hskip -0.05cm{\mathcal C}\dot Q+\;^{t}\hskip-0.06cm{\mathcal K}Q=0,\;Q(0)=\delta \Phi_0,\;\dot Q(0)=\delta \Phi_1.\eeq
By multiplying by $Q$ the state equation satisfied by $X$ and by integrating from $0$ to $T$ we obtain:
\beq\label{eq810}\hskip-0.5cm\left \{\begin{array}{l}(\Lambda (\Phi),\delta \Phi)=-({\mathcal M}\dot X(T)+{\mathcal C}X(T),Q(T)) +({\mathcal M} X(T),\dot Q(T))+\\\\({\mathcal M}\dot X(0)+{\mathcal C}X(0),\delta\Phi_0)-({\mathcal M} X(0),\delta\Phi_1)+\intj_0^T({\mathcal F}(s),Q(s))ds.\end{array}\right. \eeq
We set:
\beq\label{eq820}\hskip-0.5cm\left \{\begin{array}{l}L\in \a R^{2N},\;\hbox{ such that }:\;\forall \delta \Phi\in \a R^{2N},\\\\(L,\delta \Phi)\hskip-0.05cm=\hskip-0.05cm\intj_0^T\hskip-0.05cm({\mathcal F}(s),Q(s))ds\hskip-0.05cm+\hskip-0.05cm({\mathcal M}\dot X(0)+{\mathcal C}X(0),\delta\Phi_0)\hskip-0.05cm-\hskip-0.05cm({\mathcal M} X(0),\delta\Phi_1).\end{array}\right.\eeq
From Theorem \ref{th10} we can define uniquely $\Phi\in \a R^{2N}$ such that:
\beq\label{eq830}\Lambda (\Phi)=L.\eeq
With this choice of $\Phi$ we deduce from (\ref{eq810}) that:
\beq\label{eq840}\forall \delta \Phi\in \a R^{2N},\;-({\mathcal M}\dot X(T)+{\mathcal C}X(T),Q(T)) +({\mathcal M} X(T),\dot Q(T))=0,\eeq
One can choose for instance $\delta \Phi=(Q(0),\dot Q(0))\in \a R^{2N}$ where $Q$ is solution of the retrograde equation:
\beq\label{eq850}{\mathcal M}\ddot Q-\;^{t}\hskip -0.05cm{\mathcal C}\dot Q+\;^{t}\hskip-0.06cm{\mathcal K}Q=0,\;Q(T)=-{\mathcal M}\dot X(T)+{\mathcal C}X(T),\;\dot Q(T)={\mathcal M} X(T),\eeq
which implies that $X(T)=\dot X(T)=0$. Hence the choice of $\Phi$ solution of (\ref{eq830}) leads to an exact control $u(\Phi)$.
\subsubsection{Convergence of $u^\varepsilon$ to $u^0$.} First of all from the definition of the optimal control model one has:
\beq\label{eq900}\forall v\in {\Big [}H^1_0(]0,T[){\Big ]}^p,\;\;J^\varepsilon(u^\varepsilon)\leq J^\varepsilon (v).\eeq
Choosing $v=u^0$, which is an exact control, leads to the two inequalities:
\beq\label{eq810}\hskip-.3cm\left \{\begin{array}{l} \vert  X^\varepsilon (T)\vert_{\a R^{2N}}^2+\vert  \dot X^\varepsilon (T)\vert_{\a R^{2N}}^2\leq \varepsilon \intj_0^T{\Big [}\alpha \vert u^0 \vert_{\a R^p}^2+\beta\vert\fracj{du^0}{dt}\vert_{\a R^p}^2+2\gamma \vert\fracj{du^0}{dt}\vert_{\a R^p}{\Big ]},\\\\ \hbox{and}\\\\
\hskip-.2cm \intj_0^T\hskip-.1cm{\Big [}\alpha \vert u^\varepsilon \vert_{\a R^p}^2\hskip-.05cm+\hskip-.05cm\beta\vert\fracj{du^\varepsilon}{dt}\vert_{\a R^p}^2\hskip-.05cm+\hskip-.05cm 2\gamma \vert\fracj{du^\varepsilon}{dt}\vert_{\a R^p}{\Big ]}\hskip-.1cm \leq \hskip-.1cm \intj_0^T\hskip-.1cm{\Big [}\alpha \vert u^0 \vert_{\a R^p}^2\hskip-.05cm +\hskip-.05cm \beta\vert\fracj{du^0}{dt}\vert_{\a R^p}^2\hskip-.05cm +\hskip-.05cm 2\gamma \vert\fracj{du^0}{dt}\vert_{\a R^p}{\Big ]}.\end{array}\right.\eeq
These estimates prove that $u^\varepsilon$ is bounded in the space ${\Big [}H^1_0(]0,T[){\Big ]}^p$ versus $\varepsilon$. From the state equation satisfied by $X^\varepsilon$ (an ordinary differential equation), one can claim that this state variable is bounded  in the space ${\Big [}H^2(]0,T[){\Big ]}^N\subset {\Big [}{\mathcal C}^1([0,T]){\Big ]}^N$.\\\\ Furthermore:
$$ \limj_{\varepsilon\rightarrow 0} X^\varepsilon(T)=\limj_{\varepsilon \rightarrow 0}\dot X^\varepsilon(T)=0 \hbox{ in }\a R^N.$$
From (\ref{eq810}) one can state that there is a subsequence of $u^\varepsilon$  denoted by $u^{\varepsilon '}$ which converges weakly to an element $u^*\in {\Big [}H^1_0(]0,T[){\Big ]}^p$ (for the weak topology of this space). Hence $X^{\varepsilon '}$converges to an element $X^*$ which satisfies $X^*(T)=\dot X^*(T)=0$. Therefore $u^*$ is an exact control in ${\Big [}H^1_0(]0,T[){\Big ]}^p$. The set of these exact controls is denoted by $E_{ex}$. Using the lower-semi continuity of continuous convex functions with respect to the weak topology, we can claim from (\ref{eq810}) that:
\beq\label{eq820}\hskip-.4cm\begin{array}{l}\intj_0^T\hskip-.1cm\alpha\vert  u^*\vert_{\a R^p}^2\hskip-.1cm+\hskip-.1cm\beta \vert \fracj{du^*}{dt}\vert_{\a R^p}^2\hskip-.1cm+\hskip-.1cm2\gamma\vert \fracj{du^*}{dt}\vert_{\a R^p}\leq \hskip-.1cm\minj_{v\in E_{ex}}\intj_0^T\alpha\vert v\vert_{\a R^p}^2+\beta \vert \fracj{dv}{dt}\vert_{\a R^p}^2+2\gamma\vert \fracj{dv}{dt}\vert_{\a R^p}\\\\\leq \intj_0^T\alpha\vert u^0\vert_{\a R^p}^2+\beta \vert \fracj{du^0}{dt}\vert_{\a R^p}^2+2\gamma\vert \fracj{du^0}{dt}\vert_{\a R^p}.\end{array}\eeq
In other words, $u^*$ is an exact control which minimizes the cost of the control. Because of the strict convexity of his cost function, we can claim that $u^*$ is the unique control which minimizes the cost function.
%
%
In addition, recalling that $u^0$ is solution of (\ref{eq611}) where we choose $v=u^*$:
\beq\label{eq830}\begin{array}{l} \intj_0^T\alpha(u^0-u^*,.u^0)+\beta(\fracj{du^0}{dt}-\fracj{du^*}{dt},\fracj{du^0}{dt})+\gamma (\vert\fracj{du^0}{dt}\vert_{\a R^p}-\vert\fracj{du^*}{dt}\vert_{\a R^p})\leq \\\\\intj_0^T(\;^{t}\hskip-0.05cm{\mathcal B}Q(\Phi),(u^*-u^0)=\intj_0^T(\hskip-0.05cm{\mathcal B}(u^*-u^0),Q(\Phi))=\\\\\intj_0^T({\Big (}{\mathcal M}(\ddot X^*-\ddot X^0)+{\mathcal C}(\dot X^*-\dot X^0)+{\mathcal K}(X^*-X^0){\Big )},Q)=0,\end{array}\eeq
we obtain:
\beq\label{eq840}\begin{array}{l}\intj_0^T\alpha\vert  u^0\vert_{\a R^p}^2+\beta\vert \fracj{du^0}{dt}\vert_{\a R^p}^2+\gamma\vert\fracj{du^0}{dt}\vert_{\a R^p}\leq \\\\\intj_0^T\alpha (u^0,u^*)+\beta (\fracj{du^0}{dt},\fracj{du^*}{dt})+\gamma\vert\fracj{du^*}{dt}\vert_{\a R^p}\end{array}\eeq
and from Cauchy-Schwarz triangular inequality\footnote{$\forall a,b\in \a R,\;\theta\in \a R^{*,+},\;2ab\leq \theta a^2+\fracj{1}{\theta}b^2$}
:
\beq\label{eq850}\begin{array}{l}\fracj{1}{2}\intj_0^T\alpha\vert u^0\vert_{\a R^p}^2+\beta \vert\fracj{du^0}{dt}\vert_{\a R^p}^2+2\gamma\vert\vert\fracj{du^0}{dt}\vert\\\\\leq \fracj{1}{2}\intj_0^T\alpha\vert u^*\vert_{\a R^p}^2+\beta \vert\fracj{du^*}{dt}\vert_{\a R^p}^2+2\gamma\vert\fracj{du^*}{dt}\vert_{\a R^p}\end{array}\eeq
Finally, we proved that $u^0$ is also a minimizer of the cost function  among  the exact controls. Because of the uniqueness of the minimizer of a strictly convex function, we proved that $u^*=u^0$. Hence, from the uniqueness of the limit,  all the sequence $u^\varepsilon$ converges weakly to the control $u^0$.

\subsubsection{The strong convergence of $u^\varepsilon$ to $u^0$}
Let us observe that:
\beq\label{eq860}\begin{array}{l}\intj_0^T\alpha\vert u^\varepsilon-u^0\vert_{\a R^p}^2+\beta\vert \fracj{du^\varepsilon}{dt}-\fracj{du^0}{dt}\vert_{\a R^p}^2=\\\\\intj_0^T\alpha\vert  u^\varepsilon\vert_{\a R^p}^2+\beta\vert \fracj{du^\varepsilon}{dt}\vert_{\a R^p}^2+2\gamma \vert\fracj{du^\varepsilon}{dt}\vert_{\a R^p}-2\gamma\vert\fracj{du^\varepsilon}{dt}\vert_{\a R^p}\\\\-2\intj_0^T\alpha (u^\varepsilon,u^0)-2\beta (\fracj{du^\varepsilon}{dt},\fracj{du^0}{dt})+\intj_0^T\alpha \vert  u^0\vert_{\a R^p}^2+\beta \vert \fracj{du^0}{dt}\vert_{\a R^p}^2,\end{array}\eeq
From (\ref{eq810}):
\beq\label{eq870}\begin{array}{l}\intj_0^T\alpha\vert u^\varepsilon-u^0\vert_{\a R^p}^2+\beta \vert\fracj{du^\varepsilon}{dt}-\fracj{du^0}{dt}\vert_{\a R^p}^2\leq\intj_0^T\alpha\vert  u^0\vert_{\a R^p}^2+\beta \vert\fracj{du^0}{dt}\vert_{\a R^p}^2+2\gamma \vert\fracj{du^0}{dt}\vert_{\a R^p}\\\\-2\gamma\vert\fracj{du^\varepsilon}{dt}\vert_{\a R^p}-2\intj_0^T\alpha (u^\varepsilon,u^0)-2\beta (\fracj{du^\varepsilon}{dt},\fracj{du^0}{dt})+\intj_0^T\alpha  \vert u^0\vert_{\a R^p}^2+\beta \vert  \fracj{du^0}{dt}\vert_{\a R^p}^2,\end{array}\eeq
and at the limit when $\varepsilon \rightarrow 0$ because of the weak convergence of $u^\varepsilon$ to $u^0$ and using the lower semi-continuity for the weak topology in the space $H_0^1(]0,T[)$ of the continuous and convex function $\intj_0^T\vert\fracj{du^\varepsilon}{dt}\vert_{\a R^p}$:
\beq\label{eq880}\limj_{\varepsilon\rightarrow 0}\intj_0^T\alpha\vert  u^\varepsilon-u^0\vert_{\a R^p}^2+\beta \vert\fracj{du^\varepsilon}{dt}-\fracj{du^0}{dt}\vert_{\a R^p}^2=0,\eeq
which proves the strong convergence of $u^\varepsilon$ to $u^0$ in the space ${\Big [}H^1_0(]0,T[){\Big ]}^p$. This completes the {\it pseudo} asymptotic analysis of the optimal control problem when the marginal cost of the control $\varepsilon \rightarrow 0$.
\subsubsection{An algorithm for computing $u^0$} One can use the algorithm used for proving the existence and uniqueness of $\Phi$ solution of (\ref{eq830}). Let us summarize it:
{\it 
\begin{enumerate}[step ]
\item 1 Let us start with $\Phi^0=0$ and $u^0=0$; $n=0$;
\item 2 Let us assume that $\Phi^n$ and $u^n$ are known. Let us compute $X(\Phi^n)$  solution of: 
$$\left \{\begin{array}{l}{\mathcal M}\ddot X(\Phi^n)+{\mathcal C}\dot X(\Phi^n)+{\mathcal K}X(\Phi^n)=F+{\mathcal B}u^n\\\\X(\Phi^n)(0)=X_0,\;\dot X(\Phi^n)(0)=X_1;\end{array}\right.$$
\item 3 Compute $P(\Phi^n)$ solution of:
$$\left \{\begin{array}{l}{\mathcal M}\ddot P(\Phi^n)-\;^{t}\hskip-0.05cm{\mathcal C}\dot P(\Phi^n)+\;^{t}\hskip-0.05cm{\mathcal K}P(\Phi^n)=0,\\\\P(\Phi^n)(0)=\Phi^n_0,\;\dot P(\Phi^n)(0)=\Phi^n_1;\end{array}\right.$$
 \item 4 Compute $u^{n+1}=u(\Phi^n)\in {\Big [}H^1_0(]0,T[){\Big ]}^p$ (using the duality algorithm of R. Glowinski with $\lambda(\Phi^n)\in B_1$) solution of:
 $$\left \{\begin{array}{l}\forall v\in {\Big [}H^1_0(]0,T[){\Big ]}^p,\\\\\intj_0^T\hskip-0.05cm\alpha (u(\Phi^n),v)\hskip-0.05cm+\hskip-0.05cm\beta(\fracj{du(\Phi^n)}{dt},\fracj{dv}{dt})\hskip-0.05cm+\hskip-0.05cm\gamma (\lambda(\Phi^n),\fracj{dv}{dt})\hskip-0.05cm=\hskip-0.05cm-\intj_0^T(\;\hskip-0.1cm^{t}\hskip-0.05cm{\mathcal B}P(\Phi^n),v),\\\\
 \forall \mu\in B_1,\;(\mu-\lambda(\Phi^n)).\fracj{du(\Phi^n)}{dt}\leq 0\end{array}\right.$$ 
 \item 5 Compute $\Lambda(\Phi^n)-L$;
 \item 6 Set: $\Phi^{n+1}=\Phi^n-\varrho[\Lambda(\Phi^n)-L)$;
 \item 7 Test the convergence to zero of $\Phi^{n+1}-\Phi^n$;
 \item 9 If no convergence replace $\Phi^n$ by $\Phi^{n+1}$ and go to the second step;
 \end{enumerate}
 }
\begin{remark}\label{rem300} In practice the convergence in $\Phi$ is very slow because an ill conditioning of the operator $\Lambda$. Hence one efficient possibility is to use a preconditioning  using the operator $\Lambda^0$ similar to $\Lambda$ but obtained for $\gamma=0$. In this case, the inverse of $\Lambda^0$ is easy to compute (or to factorize). This enables to readjust all the components of $X$ with similar values and to get a satisfying convergence. \hfill $\Box$
\end{remark}
\begin{remark}\label{rem20} The choice of $\varrho$ should be optimized at each step (see J. Cea \cite{CEA}) in order to accelerate the convergence.\hfill $\Box$\end{remark}
 \subsubsection{A brief discussion on the role of the norm of the term: $\fracj{du}{dt}$}
Let us consider a non-empty time interval $]t_1,t_2[\subset [0,T]$ such that the component $k$ of the control $u^0$ is such that one has:  $u_k^0(t_1)=u_k^0(t_2)$, (let us recall that this make sense because $H^1(]0,T[)\subset {\mathcal C}^0([0,T])$). Obviously $t_1=0$ and $t_2=T$ is a possibility.  We introduce the following virtual  control $v$ defined on the whole segment  $[0,T]$:
\beq\label{eq191}\forall j=1,p,\;j\neq k,\;v_j=u_j^0\hbox{ and }v_k=\left \{\begin{array}{l}u_k^0 \hbox{ on } ]0,t_1[;
\\\\
u^0_k(t_1)=u^0_k(t_2)=d_k\hbox{ on }[t_1,t_2];
\\\\
v=u_k^0\hbox{ on } ]t_2,T[.
\end{array}\right.\eeq
Introducing this value of $v$ in (\ref{eq611}) we obtain (the derivative of $v_k$ on $[t_1,t_2]$ is zero):
\beq\label{eq192}\intj_{t_1}^{t_2}\alpha\vert d_k-u^0_k\vert^2+\beta\vert \fracj{du_k^0}{dt}\vert^2+\gamma\vert\fracj{du_k^0}{dt}\vert\leq\intj_{t_1}^{t_2}(\;^{t}\hskip-0.05cm{\mathcal B}P^1)_k+\alpha d_k,d_k-u_k^0).\eeq
From:
$$(d_k-u^0_k)(t)=-\intj_{t_1}^t\fracj{du^0_k}{dt}(s)ds,$$
or else:
$$\intj_{t_1}^{t_2}\vert d_k-u^0_k\vert\leq (t_2-t_1)\intj_{t_1}^{t_2}\vert\fracj{du^0_k}{dt}\vert,$$
and inserting this estimate in the right-hand side of (\ref{eq192}), we deduce that:
\beq\label{eq193}\gamma\intj_{t_1}^{t_2}\vert \fracj{du^0_k}{dt}\vert\leq {\Big [}(t_2-t_1)\vert\vert (\;^{t}\hskip-0.05cm{\mathcal B}P^1)_k\vert\vert_{0,\infty,]t_1,t_2[}+\alpha d_k{\Big ]}\;\;\intj_{t_1}^{t_2}\vert \fracj{du^0_k}{dt}\vert.\eeq
But $P^1$ and $u^0$ are solution of a system depending only on the data $X_0,X_1$(the initial condition of the state variable $X$) and ${\mathcal F}$ (the right-hand side of the state equation). Hence from the previous analysis there exists a constant $c_6>0$ (which could be estimated) such that:
\beq\label{eq194}\left \{\begin{array}{l}\gamma \intj_{t_1}^{t_2}\vert \fracj{du^0_k}{dt}\vert\leq\\\\{\Big [}c_6(t_2-t_1)(\vert X_0\vert_{\a R^{2N}}+\vert X_1\vert_{\a R^{2N}}+\vert\vert {\mathcal F}\vert\vert_{0,2,]0,T[})+\alpha d_k{\Big ]}\intj_{t_1}^{t_2}\vert \fracj{du^0_k}{dt}\vert.\end{array}\right.\eeq
Hence we can state the following property:
\beq\label{eq195}\begin{array}{l} \gamma> c_6(t_2-t_1)(\vert X_0\vert_{\a R^{2N}}+\vert X_1\vert_{\a R^{2N}}+\vert\vert {\mathcal F}\vert\vert_{0,2,]0,T[})+\alpha d_k \\\\
\hskip 3cm=>\;u_k^0=d_k\hbox{ on }[t_1,t_2]\end{array}\eeq
One can observe that the coefficient $\alpha$ gives {\it a security gap} to the smoothing algorithm. It enables to expand the range given by the initial data and the right-hand side. In fact, one can say that the tuning of the method involves the adjustment of  $\alpha$ and $\gamma$ versus the initial data. \\
This property of smoothing  the control is the goal of the term in factor of $\gamma$. It cancels  or at least reduces the variations of the control between $t_1$ and $t_2$ if $\gamma$ is large enough compared to the data $X_0,X_1,{\mathcal F}$ for a given value of the parameter $\alpha$. On can  check this result on the numerical tests (see Figures \ref{fig1bis}-\ref{fig3bis}-\ref{fig4bis}). 
%
\section{Numerical approximation of the control model}
All the variables (the state function $X$, the adjoint state function $P^1$ and the control $u^0$ or the multiplier $\lambda^0$ are represented by piecewise linear function on $[0,T]$ which is split into sub-segments with the same length $\Delta t$. The solution of the differential equation (for $X$ and $P^1$) are solved by the unconditionally stable Newmark scheme \cite{NEWMARK}. It is worth to point out that even if the direct model is stable (and this is not always the case), the model  used for computing the adjoint state $P^1$, is not necessarily stable as far as the damping change of sign unless we solve it using the retrograde scheme (starting from $T$ instead of $0$). Furthermore let us underline that it can occur that the direct model is unstable because of a negative damping (stall flutter mechanism as described in E.H.  Dowell \cite{DOWELL}). Hence the choice of the Newmark parameters (see P.A. Raviart and J.M. Thomas \cite{NEWMARK})  are almost necessarily: $\beta=.25$ and $\gamma=.5$. The solution method for computing $(u^0,\lambda^0)$ when $P^1$ is given is the following one (Uzawa algorithm \cite{GLOW}):
\begin{itemize}[$\#$]
\item start from $\lambda^{0,0}=0$ and $u^{0,0}$ solution of:
$$\forall v\in {\Big [}H^1_0(]0,T[){\Big ]}^p,\;\intj_0^T\alpha (u^{0,0},v)+\beta (\fracj{du^{0,0}}{dt},\fracj{dv}{dt})=-\intj_0^T(\;^{t}\hskip-0.05cm {\mathcal B}P^1,v);$$
\item assuming $u^{0,n}$ known, we iterate $\lambda^0$ by:
$$\lambda^{0,n+1}=\lambda^{0,n}-\varrho \fracj{du^{0,n}}{dt},\hbox{ where }\varrho >0; $$
\item compute $u^{0,n+1}$ by solving:
$$\left \{\begin{array}{l}\forall v\in {\Big [}H^1_0(]0,T[){\Big ]}^p,\\\\\intj_0^T\alpha (u^{0,n+1},v)+\beta (\fracj{du^{0,n+1}}{dt},\fracj{dv}{dt})=-\gamma \intj_0^T(\lambda^{0,n+1},\fracj{dv}{dt})-\intj_0^T(\;^{t}\hskip-0.05cm {\mathcal B}P^1,v);\end{array}\right.$$
\item test on the convergence of $\lambda^{0,n+1}$  if yes => stop, else go to the next item;
\item $n=n+1$ go to the second item;
\end{itemize}
\section{Numerical tests on a simple model}\label{sect5} Our first test is a simple mechanical system with two springs and two masses. The goal is just to illustrate the behavior of the algorithm in the most simple case (see Figure \ref{fig0}) before applying it to a more complex mechanical situation where instabilities can occur due to a negative damping term.\\\\
\begin{figure}[htbp] 
   \centering
   \includegraphics[width=7cm,height=2.3cm]{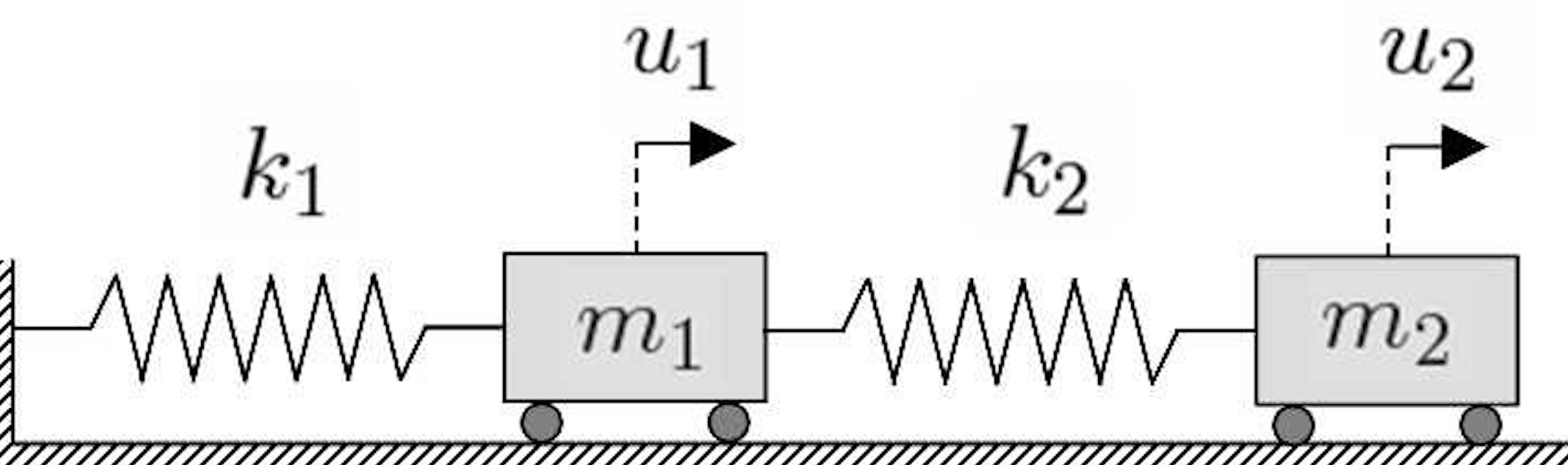} 
   \caption{The elementary model for checking the algorithm}
   \label{fig0}
\end{figure}
\\
\noindent  The left Figure \ref{fig1} represents the evolution of the displacements $x_1$ and $x_2$ without control. The right one represents the same components but with the exact control introduced in this text and for two values of the coefficient $\gamma$.  Both are plotted versus the time (see Figure \ref{fig0} for the description of the mechanical system). The time $T$ has been chosen  arbitrary fixed equal to $2.6\;s$ (approximately twice the largest period of the system) in order to avoid too many oscillations on the graphs. But similar results can be obtained for any values of $T$. The  different curves correspond to various values of the coefficient $\gamma$ in the control criterion. One can see that the variations of the two components of the state variables are reduced when $\gamma$ increases. Nevertheless the controls are always exact at time $T$.\\ The number of degrees of freedom ({\it dof}) is not meaningful, it could be larger without new difficulties. The ill conditioning is mainly due to the ratio between the values of the displacements and their derivatives. This is why a preconditioning has been used in the case of the  boat models  (see Remark \ref{rem300}), in order to speed up the convergence of the algorithm in the computation of $\Phi$. In the applications to the flying boats that we consider in this text, two {\it dof} are sufficiently representative of the physical behavior of the system (heaving and pitching) this is why we restrict our tests to two {\it dof} as the authors of \cite{LUNA} have suggested. 
\begin{figure}[htbp] 
   \centering
   \includegraphics[width=6.2cm,height=5cm]{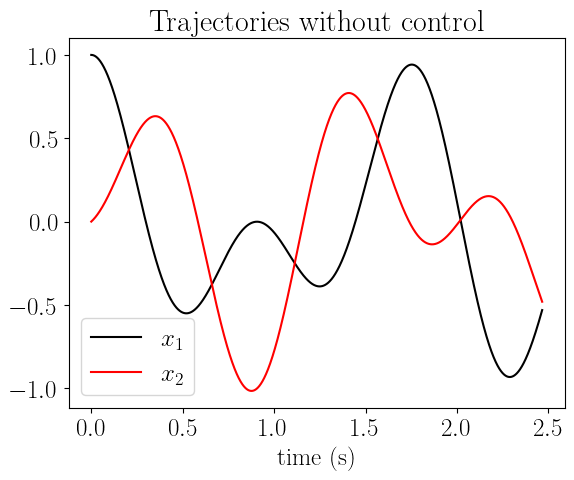}  \includegraphics[width=6.2cm,height=5cm]{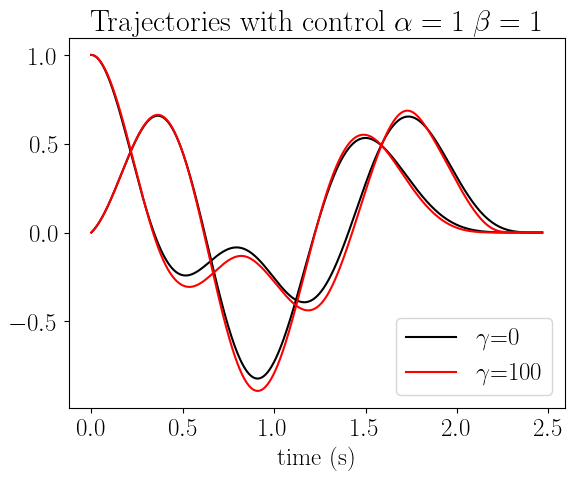} 
   \caption{Solutions for the two {\it dof} simple system without control (left) and with one control (right)}
   \label{fig1}
\end{figure}
\\\\
The control function (only one control is used in this simple example and it is applied to the first mass) are plotted on Figures \ref{fig1bis}. One can see the effect of $\gamma$ which reduces the variations of the control (due to the {\it term which restricts the variations of the control}). And this was precisely the goal of this study.
\begin{figure}[htbp] 
   \centering
   \includegraphics[width=6.1cm,height=5.1cm]{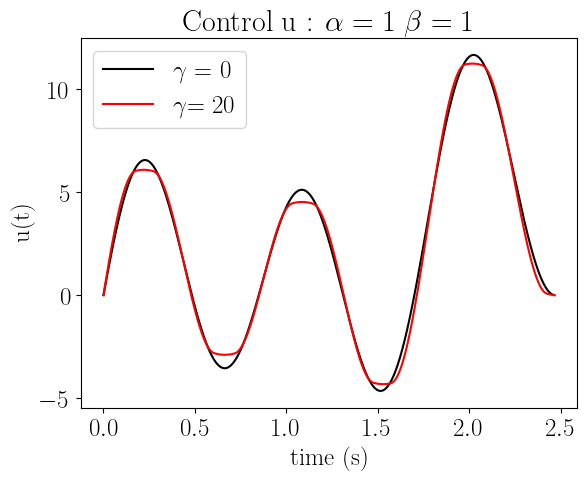} \includegraphics[width=6.1cm,height=5.1cm]{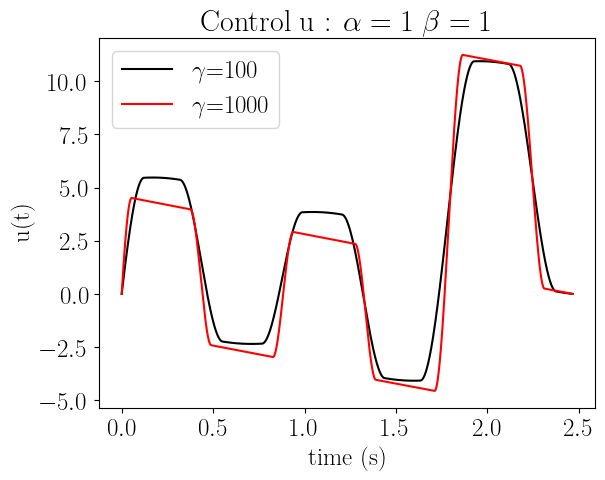} 
   \caption{Control functions for several values of $\gamma$ : 0, 20,100 and 1000   }
   \label{fig1bis}
\end{figure}
\\\\
The flattening of the control is prescribed by the term involving the first order derivative of the control. It appears on the numerical tests that there is an optimal value for the coefficient $\gamma$ for $\beta$ and $\alpha$ given. But these coefficients should be tuned once for all in practical applications. It is remarkable that for $\gamma$ important the exact control looks like (almost) a bang-bang control that one obtains when the time $T$ is sufficiently small  and for control bounds given. The reduction of the oscillations on the control has two positive effects: one is the reduction of energy consumed by the regulation loop; the second one is a reduction of the fatigue of the control device.
\section{Numerical tests for a simplified boat model}\label{section6} The description of the model in given in the Annex 1. This leads to the matrices ${\mathcal M},\;{\mathcal C},\;{\mathcal K}$ and ${\mathcal B}$ which derived from a mechanical modeling, but also -and may be mainly- to the expression of the hydrodynamical forces and their derivatives versus $X$ and $\dot X$. The vector $X$ is in $\a R^2$ and represent the heaving and the pitching of the boat (see Figure \ref{figA3}).  In this case we have two controls. One is a flap at the rear of the main foils and the other one is also a flap positioned at the rear foil fixed on the rudder. Due to the leverage the amplitude of the rear control is less than the one of the main foil. But it is necessary as far as there can be singularities due to the damping matrix ${\mathcal C}$. This is a strange phenomenon which is scarcely mentioned in the engineering publications  but which can be determinant in a control loop (see E. H. Dowell \cite{DOWELL} and Ph. Destuynder-C. Fabre \cite{DESTU}). It can be at the origin of negative damping sometimes interpreted as a stall flutter and has been at the origin of several problems in aeronautics.
\\
\begin{figure}[htbp] 
   \centering
   \includegraphics[width=6.1cm,height=5cm]{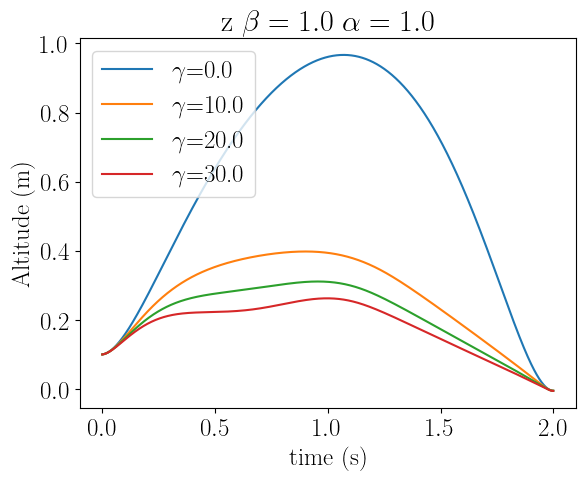}\includegraphics[width=6.1cm,height=5cm]{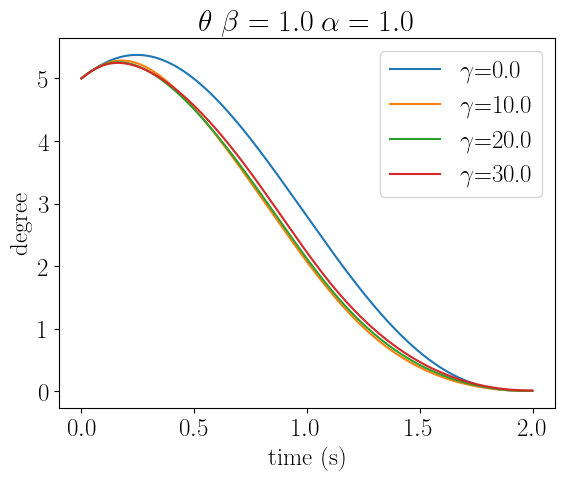} 
   \caption{Controlled trajectories for the LaSIE boat (T=1s) }
   \label{fig3}
\end{figure}
\begin{figure}[htbp] 
   \centering
   \includegraphics[width=6.1cm,height=5.2cm]{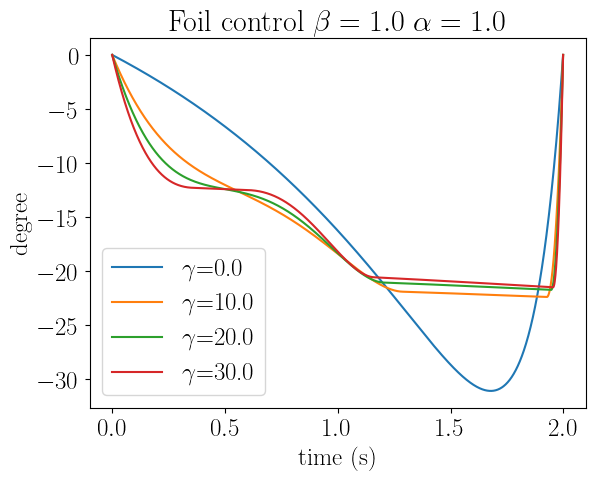}\includegraphics[width=6.1cm,height=5cm]{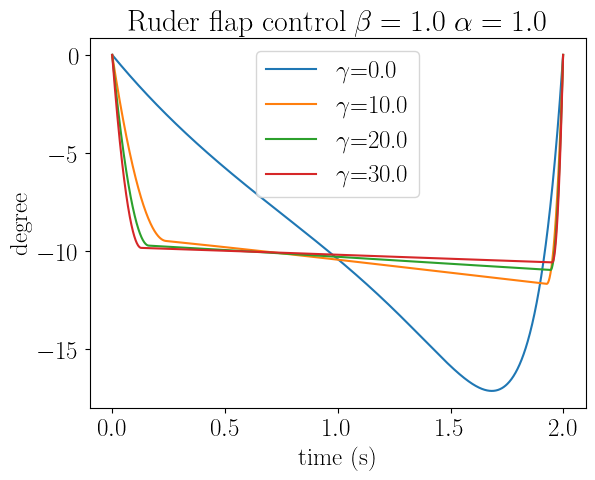} 
   \caption{Controls for the LaSIE boat (T=1s)}
   \label{fig3bis}
\end{figure}
\begin{figure}[htbp] 
   \centering
   \includegraphics[width=6.1cm,height=5.2cm]{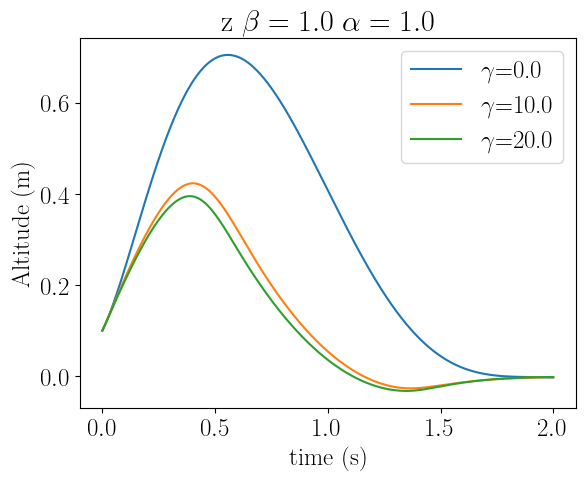}\includegraphics[width=6.1cm,height=5cm]{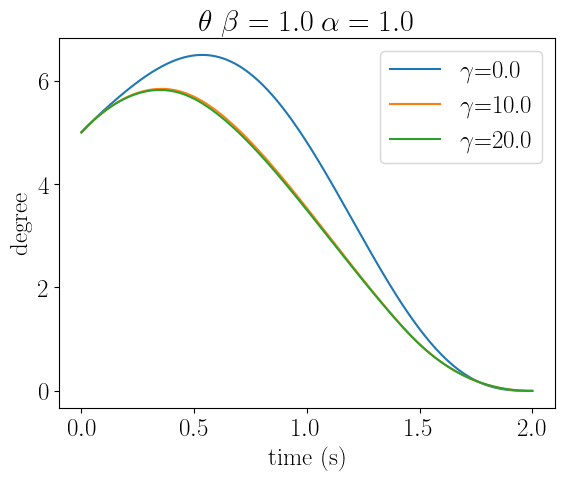} 
   \caption{Controlled trajectories for the LaSIE boat (T=2s)}
   \label{fig4}
\end{figure}
\begin{figure}[htbp] 
   \centering
   \includegraphics[width=6.1cm,height=5.2cm]{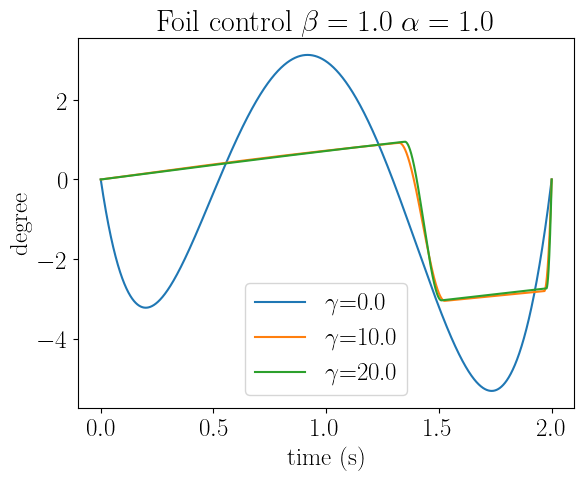}\includegraphics[width=6.1cm,height=5cm]{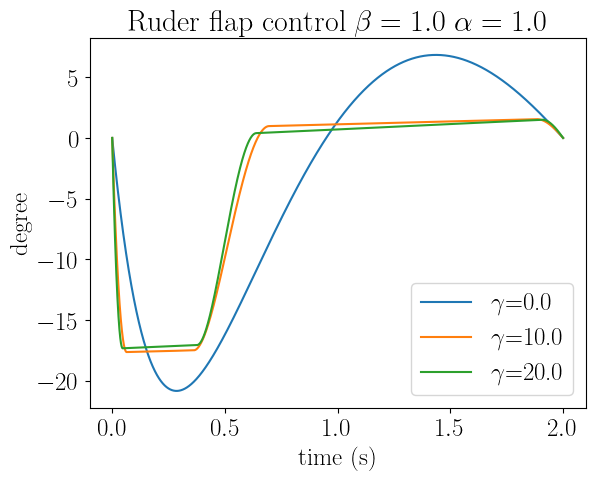} 
   \caption{Controls for the LaSIE boat (T=2s)}
   \label{fig4bis}
\end{figure}
\noindent The effect of $\gamma$ is very significant on this example. The tests represented on Figures \ref{fig3}-\ref{fig4} for the heaving and the pitch angle are associated to the two controls drawn on Figures \ref{fig3bis}-\ref{fig4bis}. The curves have been plotted for different values of the coefficient $\gamma$ on each graph and for two kinds of initial perturbations: a heaving impact (variation of the wind velocity and therefore of the lift on Figures \ref{fig3}-\ref{fig3bis}) and a pitching impact (the boat hits a large wave on Figures \ref{fig4}-\ref{fig4bis}). The results on the control (rudder and foil) are spectacular and even if the control remains exact at time $T$, we obtain a control almost bang-bang for $\gamma$ large enough. Here again the total variation of the control is less than in the case where $\gamma=0$. This reaches the goal announced for this study. One can also combine different initial conditions simultaneously which would lead to similar results. \\

\begin{remark}\label{remsuralphabetagamma}
Concerning Figures \ref{fig4} and \ref{fig4bis} and for  values of $\gamma$ larger than  $30$, the curves have not been plotted because the algorithm did not converge sufficiently fast, according to the criteria we have prescribed (convergence of the control criterion and number of iterations). Note that for a set of values $\alpha = 1$,  $\beta = 5$, the algorithm works for much larger value of the parameter $\gamma$. As suggested at the end of Section \ref{sect5}, the coefficients $\alpha$, $\beta$ and $\gamma$ are parameters that should be adjusted. There is certainly an optimal set of values for which the control will be optimal from both a quality and computational point of view. This aspect will be investigated in future works. The aim of the current one was to suggest a regularizing control term that reduces oscillations. However, when we look at the results of Figures \ref{fig3} and \ref{fig3bis}, for a time $T=1s$, we notice that choosing $\gamma =30$ instead of $\gamma =20$ does not really change the controls. Furthermore, we notice that passing a threshold, increasing $\gamma$ does not really improve the efficiency of the control.\hfill $\Box$
\end{remark}
\noindent It would be certainly more realistic to introduce at least the rolling in the model. But the comparaisons with the results concerning Luna Rossa \cite{LUNA} are not possible because the authors of this article just kept the two {\it dof}:  heaving and pitching. Anyway a more industrial model can be used for customer applications. Nevertheless, the tuning of the sails could also be included as far as one has stored enough aerodynamics tables obtained from the experiment (real time practice, wind tunnel tests or computations). The full coupling with the {CFD}\footnote{Computing fluid dynamics} seems still to be out of reach for this situation.
\section{Tests for the Luna Rossa model from the article \cite{LUNA}}\label{section7} 
The boat studied in the article \cite{LUNA} is a catamaran as those used in the America'cup 2013 and 2017. We took the matrices included in the article and we applied our control strategy. On Figures \ref{fig5} we have plotted the results obtained for the heaving and the pitching. for various values of the parameter $\gamma$. One can see that the benefit of the algorithm with $\gamma>0$ is truly significant for the heaving (left Figure \cite{fig5}). The initial conditions can be observed on the curve (initial velocity on the heaving corresponding to a wind increase and initial perturbation on the pitching for both the initial angle and its velocity corresponding to a wave encountered by the boat.
\\
The control computed for various values of $\gamma$ are plotted on Figure \ref{fig5bis}. The control for the main foil  are plotted on the left Figure \ref{fig5bis} and those for the ruder foil on the right one. For each control we can see that the benefit of $\gamma$ is real and restricts  the variations of the control. Therefore the goal that was initially stated (reduction of variations for the controls) is reached by this algorithm.
\begin{figure}[htbp]
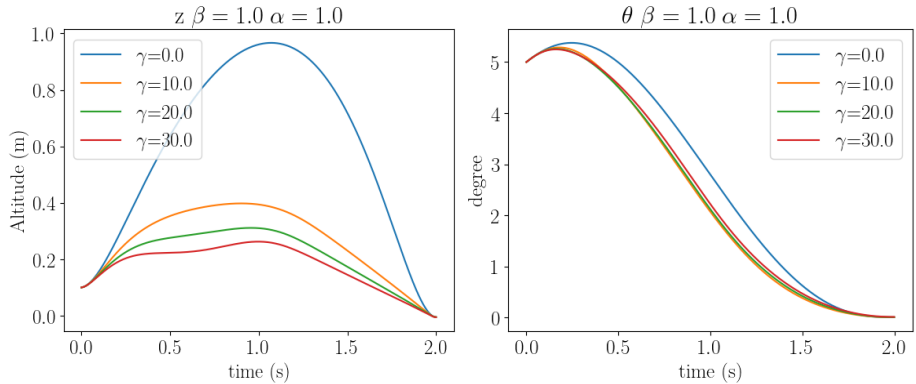
 
   \centering
   \includegraphics[width=6.1cm,height=5.2cm]{X1b1.0a1.0.png}\includegraphics[width=6.1cm,height=5.2cm]{X2b1.0a1.0.png} 
   \caption{Controlled trajectories for Luna Rossa (T=2s) }
   \label{fig5}
\end{figure}
\begin{figure}[htbp]
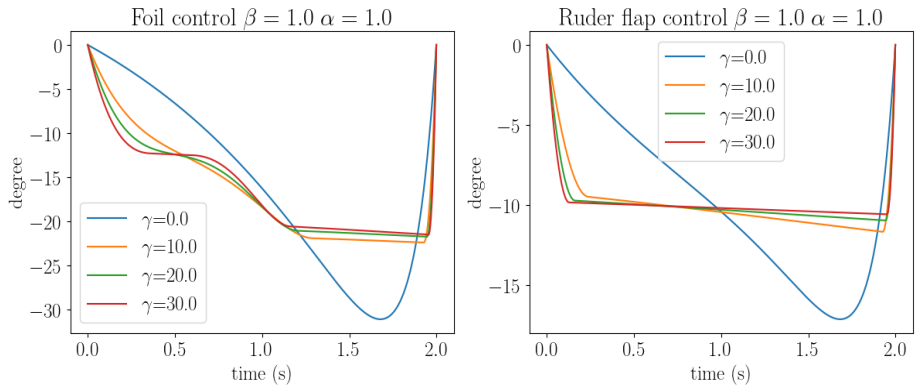
 
   \centering
   \includegraphics[width=6.1cm,height=5.2cm]{U1b1.0a1.0.png}\includegraphics[width=6.1cm,height=5.2cm]{U2b1.0a1.0.png} 
   \caption{Controls for  Luna Rossa (T=2s)}
   \label{fig5bis}
\end{figure}
\section{Conclusions} In this text we have studied a new kind of algorithms for the control of dynamical mechanical systems. The main point was to try to reduce the variations of the control. This could spare energy spent by the actuators and can also reduce the fatigue phenomenon for the control devices, mainly for the engine in charge of the control loop. The idea was to use a Tykhonov strategy \cite{TYKHONOV1}-\cite{TYKHONOV2} by adding a term in the cost functional of the optimal control problem  which is close to the total variations of the control. Furthermore, when the marginal cost of this functional (the famous $\varepsilon$ term \cite{LIONS1}) tends to zero,  we obtained an exact control as far as the classical controllability assumption is satisfied \cite{LIONS2}. The main difficulty came from the fact that the criterion is no more derivable and because of its similarity with the Bingham model for non Newtonian models as presented by G. Duvaut and J.L. Lions  \cite{DUVAUT}, we adopted a strategy close to the one suggested by R. Glowinski \cite{GLOW}. The method has been checked on a very simple model in a first step and then to a two degrees of freedom model for flying sailing boats. One (our model named the LaSIE model) is fully described in the Appendix of this  text and the other one has been developed by a research team \cite{LUNA} which has worked for the Luna Rossa challenge in the America's cup 2017. The results suggest that the strategy works well and this would justify to be extended to more complex systems in flight dynamics, mainly for instable vehicles (rockets, flying boats and aircrafts) for which there are many high frequency oscillations of the control systems..
 
\vskip 1cm
 \appendix
 \title{\Large APPENDIX }\\\\
 {\Large A simple flying boat model with two degrees of freedom}
 \\

\section{The mechanical principle of the model and the static equilibrium}
We use the orthonormal system of axis centered at point $G$ (the center of mass of the boat) represented on Figure \ref{figA3}. But this frame does not depend on neither the heaving $z$ nor the pitching angle $\theta$. It is just in a translation movement at the velocity $V$ in the direction $-{\bold e}_x$ (hence a Galilean system of axis). The water flow velocity in the previous frame attached to the boat at a translation velocity $V$, is therefore ${\bold v}_w=V{\bold e}_x$ (opposite to the movement of the boat when $z=0$ and $\theta=0$). In this frame the movement of the boat is parametrized  by the couple $(z,\theta)$ where $z$ is the heaving and $\theta$ is the pitching angle (rotation along ${\bold e}_y$). The velocity of the point $G$ in the frame used is ${\bold v}(G)=\dot z {\bold e}_z$. %

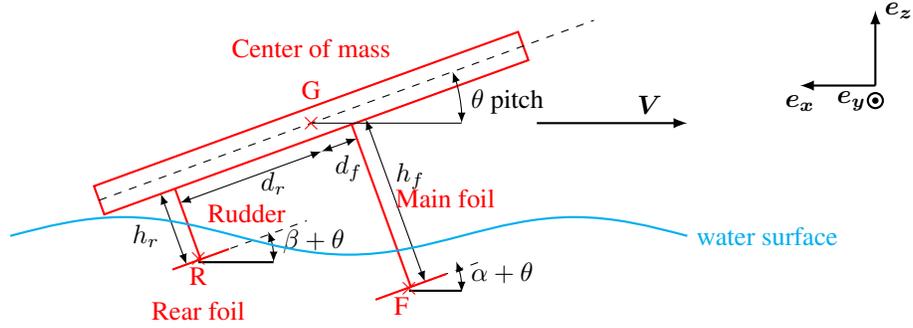
\begin{figure}[!htbp]
 \begin{center}
     \begin{tikzpicture}
     \draw[thick,-latex] (11.5,3.5)--(10.5,3.5);
     \node [below] at (10.5,3.5)  {$\boldsymbol{e_x}$};
     \draw[thick,-latex] (11.5,3.5)--(11.5,4.5);
     \node [right] at (11.5,4.5)  {$\boldsymbol{e_z}$};
     \draw[thick,fill=none](11.5,3.3) circle (0.1);
     \draw[fill=black](11.5,3.3) circle (1 pt);
     \node [left] at (11.5,3.3)  {$\boldsymbol{e_y}$};
     \draw[thick,red, rotate around={20:(4,3)}] (1,2.8) rectangle (7,3.2);
     \draw[dashed,rotate around={20:(4,3)}] (1,3) -- (8,3) ;
     \draw (4,3) -- (6,3) ;
     \draw [latex-latex](6,3) arc (0:20:2) ;
     \node [above right] at (6,3)  {$\theta$ pitch};
     \draw[thick,red, rotate around={20:(4,3)}] (2,2.8) -- (2,1.8);
	  \node [right, color=red] at (2.5,1.8) {Rudder};
	  \draw[latex-latex, rotate around={20:(4,3)}] (1.8,2.8) -- (1.8,1.8);
	  \node [right] at (1.5,1.5) {$h_r$};    
     \draw[dashed,rotate around={20:(4,3)}] (2,1.8) -- (3.5,1.8) ;
     \draw[thick,red, rotate around={20:(4,3)}] (1.6,1.8) -- (2.4,1.8) ;
     \draw[thick] (2.5,1.15) -- (3.5,1.15) ;
     \draw [latex-latex](3.5,1.15) arc (0:20:1.2) ;
    \node [above right] at (3.5,1.15)  {$\beta+\theta$};
     \draw[thick,red, rotate around={20:(4,3)}] (4.5,2.8) -- (4.5,0.5);
     \draw[latex-latex, rotate around={20:(4,3)}] (4.7,2.8) -- (4.7,0.5);
     \node [right, color=red] at (5,2) {Main foil};
     \node [right] at (5,2.3) {$h_f$};      
     \draw[dashed,rotate around={20:(4,3)}] (4.5,0.5) -- (5.5,0.5) ;
     \draw[thick,red, rotate around={20:(4,3)}] (4,0.5) -- (5,0.5) ;
    \draw [thick, color=cyan] (0,1.5) sin (1.5,1.75) cos (3,1.5) sin (4.5,1.25) cos (6,1.5)sin (7.5,1.75) cos (9,1.5);
    \node [right, color=cyan] at (9,1.5)  {water surface};
    \draw[latex-latex, rotate around={20:(4,3)}] (4,2.6) -- (4.5,2.6) ;
    \node at (4.5,2.4)  {$d_f$};
    \draw[latex-latex, rotate around={20:(4,3)}] (2,2.6) -- (4,2.6) ;
    \node  at (3.5,2.2)  {$d_r$};
     \node at (4,3)  {\textcolor{red}{$\times$}};
     \node [above] at (4,3.2)  {\textcolor{red}{G}};
     \node at (4,4)  {\textcolor{red}{Center of mass}};
     \draw[thick] (5.3,0.77) -- (6,0.77) ;
     \draw [latex-latex](6,0.75) arc (0:20:1.2) ;
     \node [above right] at (6,0.75)  {$\alpha+\theta$};
     \draw[thick,-latex] (7,3) -- (9,3) ;
     \node [above] at (8.5,3)  {$\boldsymbol{V}$};
     \node at (5.3,0.8)  {\textcolor{red}{$\times$}};
     \node [below] at (5.2,0.8)  {\textcolor{red}{F}};
     \node at (2.5,1.2)  {\textcolor{red}{$\times$}};
     \node [below] at (2.5,1.2)  {\textcolor{red}{R}};
     \node at (2.5,0.5)  {\textcolor{red}{Rear foil}};
\end{tikzpicture}
      \end{center}
   \caption{The notations used for the LaSIE boat}
   \label{figA3}
\end{figure}

\noindent First of all let us give a few indications on the foil used in this text. 
We choose a NACA 0012 \cite{NACA0012} which is a symmetrical airfoil (see Figure \ref{figA2}). The Eiffel coefficients for the lift and the drag are respectively estimated by the following  formulae where $\zeta$ is the angle of attack of the foil expressed in radian. One has:
\beq\label{eq35}\left \{\begin{array}{l}
c_L(\alpha)=(18/\pi)\zeta,\\\\
c_D(\alpha)=0.01(18/\pi)\zeta.\end{array}\right.
\eeq
 \begin{figure}[htbp] 
    \centering
    \includegraphics[width=10.cm,height=8.cm]{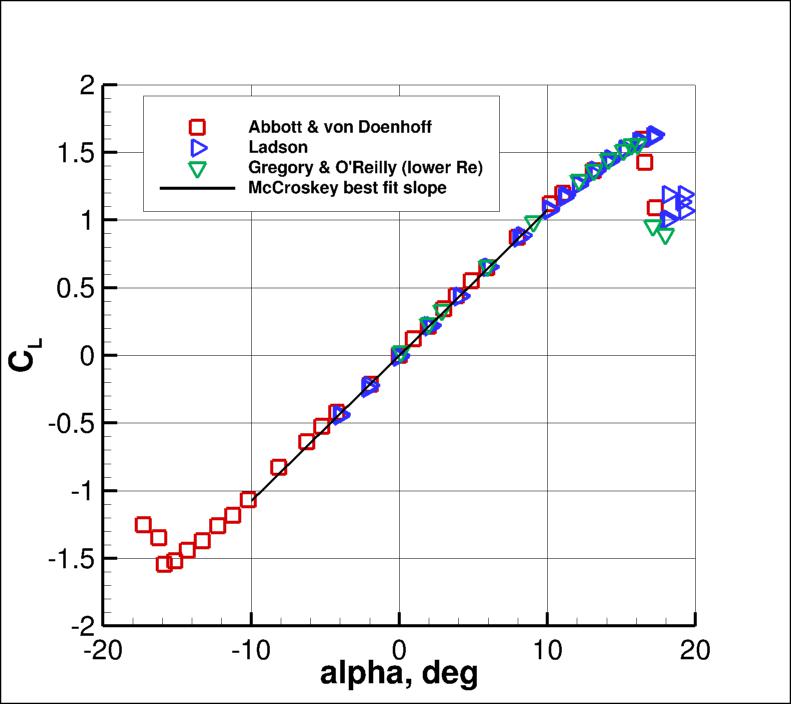} 
    \caption{Hydrodynamic lift for a NACA0012}
    \label{figA2}
 \end{figure}

\noindent The thrust force on the foils has (at least) two components. One is the drag which slow down the boat but which is quite reduced for small angles of attack of the foils. The other one is the lift which enables the boat to fly. At the equilibrium (no dynamics) this  vertical hydrodynamical forces equilibrates the weight of the boat and the tuning of the flaps are such that the pitching moment is zero at the center of mass $G$. This leads to the right values   for the angles of attack of the two foils: $\alpha_0$ for the main foil and $\beta_0$ for the rear foil fixed on the rudder. This equilibrium is traduced by the two relations written at equation (\ref{eq40}) where $S_r$ (respectively $S_f$) is the surface of the  immersed rear foil (respectively main foil). Furthermore $d_r$ and $d_f$ are defined on Figure \ref{figA3} ($d_f>0$ if $G$ is behind the main foil and $\varrho$ is the mass density of the water i.e. $1000kg/m^3$, $M$ is the mass of the boat and $g=9.81m/s^2$ the gravity acceleration):
 \beq\label{eq40}\left\{\begin{array}{l}
 \fracj{1}{2}\varrho V^2[S_rc_z(\beta_0)+S_fc_z(\alpha_0)]=Mg,
 \\\\
 \fracj{1}{2}\varrho V^2[-S_rd_rc_z(\beta_0)+S_fd_fc_z(\alpha_0)]=0.
 \end{array}\right. \eeq
 A simple computation leads to the following expressions (choosing the expression of the hydrodynamic coefficients given at (\ref{eq35}) and the angles $\alpha_0$ and $\beta_0$ are given in radian):
 \beq\label{eq50}\alpha_0=\fracj{36 Mg d_r}{\pi\varrho V^2S_f(d_r+d_f)},\;\;\;\beta_0=\fracj{36 Mgd_f}{\pi\varrho V^2S_r(d_r+d_f)}.\eeq
 %
\section{The dynamical model used for the control} 

The main point is to define a model which enables one to control the movement of the boat. Initially (in the America's cup of 2013) the whole main foil was rotating around a fixation point located on the so-called {\it daggerboard} which was at the level of the two hulls (it was a catamaran) and therefore outside of the water. For  recent America's cup boats the foils are equiped with rear flaps like on the wings of an airplane. But for sake of brevity in the mathematical formulation, we consider in this paper that  the foils are rotating around the attached point at the extremity of the arms bearing them. Concerning the main foil this arm plays also the role of a drift and in more complex models the roll angle of the foil is a control variable in order to reduce the yaw. The two movements considered in this simple example are only the heaving $z$ at the center of mass $G$  and the pitch angle  $\theta$ which is the rotation around the axis $e_y$ as said before. But the extension to more realistic engineering model would not change anything in the approach given in the following. Nevertheless, this would lead to more complex formula which are not necessary for the understanding but obviously highly required for a real use of the control method introduced in this text.
 If $M$ is a point on the boat, its velocity is:
  $${\bold v}(M)=\dot z {\bold e}_x+\dot \theta \;{\bold e}_y \wedge GM.$$
  Therefore the velocity of the center $R$ of the lifting plan of the rudder is: 
$${\bold v}(R)=(\dot z-d_r\dot \theta ){\bold e}_z-h_r\dot \theta\; {\bold e}_x$$
 and at point $F$ which is the center of the airfoil of the main foil: 
 $${\bold v}(F)=(\dot z+d_f\dot \theta){\bold e}_z-h_f\dot \theta {\bold e}_x.$$
 Hence the apparent flow velocity at $R$ (respectively $F$) is:
 \beq\label{eq30}{\bold v}^{a}(R)=V{\bold e}_x-{\bold v}(R)=(V+h_r\dot \theta){\bold e}_x+(d_r\dot \theta -\dot z){\bold e}_z,\eeq
 and respectively:
 \beq\label{eq31}{\bold v}^{a}(F)=V{\bold e}_x-{\bold v}(F)=(V+h_f\dot \theta){\bold e}_x -(d_f\dot \theta +\dot z){\bold e}_z.\eeq
 We deduce the expression of the modulus of the apparent velocity at points $R$ and $F$:
 \beq\label{31}\left\{\begin{array}{l}\vert\vert {\bold v}^{a}(R)\vert\vert^2=(V+h_r\dot \theta)^2+(\dot z-d_r\dot\theta)^2,\\\\
 \vert\vert {\bold v}^{a}(F)\vert\vert^2=(V+h_f\dot \theta)^2+(\dot z+d_f\dot\theta)^2.\end{array}\right.
 \eeq
 We use the velocities at points $R$ and $F$ for estimating the apparent angle of attack of the two foils. \\ %
 %
\begin{figure}[htbp] 
\begin{center}
     \begin{tikzpicture}
     \draw[thick,-latex] (6,4)--(5,4);
     \node [above] at (5,4)  {$\boldsymbol{e_x}$};
     \draw[thick,-latex] (6,4)--(6,5);
     \node [right] at (6,5)  {$\boldsymbol{e_z}$};
     \draw[thick,fill=none](6,4) circle (0.1);
     \draw[fill=black](6,4) circle (1 pt);
     \node [below right] at (6,4)  {$\boldsymbol{e_y}$};
     \draw (2,4)--(9,4);
     \draw[latex-latex] (2,4.5)--(6,4.5);
     \node[above] at (4,4.5) {$V+h_r \dot{\theta}$};
     \draw[rotate around={20:(6,4)}] (1.75,4) -- (10,4);
     \draw (2,4)--(2,2.5); 
     \draw(2,3.8) rectangle (2.2,4);
     \draw[-latex] (6,4)--(2,3.2);
     \node at (2.3,3.5) {$V^a_R$};
     \draw[latex-latex] (1.8,4)--(1.8,3.2);
     \node at (1,3.6) {$\dot{z} -d_r \dot{\theta}$};
     \draw [-latex](3.5,4) arc (180:200:2.5) ;
     \node at (3.2,3.8) {$\beta$};
     \draw [-latex](8,4) arc (0:20:2) ;
     \node at (8.5,4.2) {$\beta$};
     \draw [-latex](2.8,3.35) arc (190:200:3) ;
     \node at (2.5,2.9) {$\beta^a$};
      \node at (1,5) {Rear foil};
        \node at (9,3) {$\textrm{tan} \left(\beta - \beta^a\right)=\dfrac{\dot{z} -d_r \dot{\theta}}{V+h_r \dot{\theta}}$};
         \draw (1,2.4)--(10,2.4);
       %
     \draw[thick,-latex] (6,1)--(5,1);
     \node [above] at (5,1)  {$\boldsymbol{e_x}$};
     \draw[thick,-latex] (6,1)--(6,2);
     \node [right] at (6,2)  {$\boldsymbol{e_z}$};
     \draw[thick,fill=none](6,1) circle (0.1);
     \draw[fill=black](6,1) circle (1 pt);
     \node [below right] at (6,1)  {$\boldsymbol{e_y}$};
     \draw (2,1)--(9,1);
     \draw[latex-latex] (2,1.5)--(6,1.5);
     \node[above] at (4,1.5) {$V+h_f \dot{\theta}$};
     \draw[rotate around={20:(6,1)}] (1.75,1) -- (9,1);
     \draw [latex-] (2,1)--(2,-1); 
     \node at (1.8,0) {$V^a_f$};
     \draw[-latex] (6,1)--(2,0.2);
     \draw[latex-latex] (1.8,1)--(1.8,0.2);
     \node at (1,0.6) {$\dot{z} +d_f \dot{\theta}$};
     \draw [-latex](3.5,1) arc (180:200:2.5) ;
     \node at (3.2,0.8) {$\alpha$};
     \draw [-latex](8,1) arc (0:20:2) ;
     \node at (8.5,1.2) {$\alpha$};
     \draw [-latex](2.8,0.35) arc (190:200:3) ;
     \node at (2.5,-0.1) {$\alpha^a$};
      \node at (1,2) {Main foil};
        \node at (9,0) {$\textrm{tan} \left(\alpha - \alpha^a\right)=\dfrac{\dot{z} +d_f \dot{\theta}}{V+h_f \dot{\theta}}$};

      \end{tikzpicture}
      \end{center}
    \caption{Diagram of the apparent velocity at points $R$ or $F$ of the foils for LaSIE boat}
    \label{figA5}
 \end{figure}
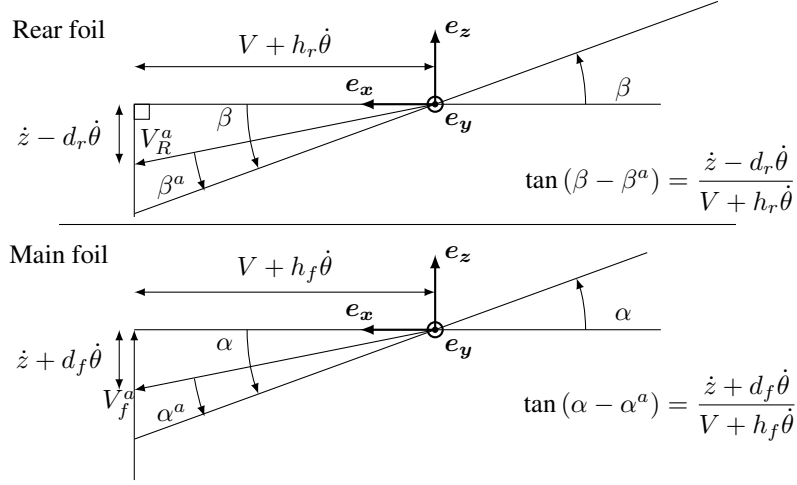
 According to Figure \ref{figA5} one has the folllowing expressions:\\
 - first of all, by taking the derivative of $\beta^{a}$ (respectively $\alpha^{a}$) with respect to the kinematical parameters $\dot z$ and $\dot \theta$ one obtains:
 \beq\label{eq60}\begin{array}{l}\fracj{\partial \beta^{a}}{\partial \dot z}=-\fracj{1}{V},\;\;\;\fracj{\partial \beta^{a}}{\partial \dot \theta}=\fracj{d_r}{V},\;\;
 \fracj{\partial \alpha^{a}}{\partial \dot z}=-\fracj{1}{V},\;\;\;
 \fracj{\partial \alpha^{a}}{\partial \dot \theta}=-\fracj{d_f}{V}.\end{array}\eeq

- then for the derivatives of the modulus of the apparent velocities at $R$ and $F$:
 \beq\label{70}{\begin{array}{l}\fracj{\partial \vert\vert {\bold v}^{a}(R)\vert\vert^2}{\partial \dot z}=0,\;\;\;\fracj{\partial \vert\vert {\bold v}^{a}(R)\vert\vert^2}{\partial \dot \theta}=2h_rV,\\\\
 \fracj{\partial \vert\vert {\bold v}^{a}(F)\vert\vert^2}{\partial \dot z}=0,\;\;\;\fracj{\partial \vert\vert {\bold v}^{a}(F)\vert\vert^2}{\partial \dot \theta}=2h_fV.\end{array}}\eeq
 The next step consists in the linearization of the forces applied to the boat through the two foils (main and rudder). Let us make explicit the expression of the heaving force and the one of the pitching moment at the center of mass of the boat:
 \beq\label{eq80}\hskip -.5cm{\begin{array}{l}{\mathcal F}_z=-Mg+\fracj{\varrho}{2} {\Big [}S_r\vert\vert {\bold v}^{a}(R)\vert\vert^2c_z(\beta^{a}+\theta+\delta_r)+S_f\vert\vert {\bold v}^{a}(F)\vert\vert^2c_z(\alpha+\theta+\delta_f){\Big ]},\\\\
 {\mathcal M}_y=\fracj{\varrho}{2}{\Big [}-\vert\vert {\bold v}(R)\vert\vert^2S_rd_rc_z(\beta^{a}+\theta+\delta_r)+\vert\vert {\bold v}(F)\vert\vert^2S_fd_fc_z(\alpha^{a}+\theta+\delta_f){\Big ]}.
 \end{array}}\eeq
 The formal linearization around $\theta=0,\;z=z_0$ (which is arbitrary) $\alpha=\alpha_0,\;\beta=\beta_0$, leads to (terms of order zero disappear because of the values chosen for $\alpha_0$ and $\beta_0$):
 \beq\label{eq80}{ \begin{array}{l}{\mathcal F}_z^L=\fracj{\partial {\mathcal F}_z}{\partial z}z+\fracj{\partial {\mathcal F}_z}{\partial \dot z}\dot z+\fracj{\partial {\mathcal F}_z}{\partial \theta}\theta+\fracj{\partial {\mathcal F}_z}{\partial \dot \theta}\dot \theta+\fracj{\partial{\mathcal F}_z}{\partial \delta_r}\delta_r+\fracj{\partial{\mathcal F}_z}{\partial \delta_f}\delta_f,\\\\
 {\mathcal M}_z^L=\fracj{\partial {\mathcal M}_z}{\partial z}z+\fracj{\partial {\mathcal M}_z}{\partial \dot z}\dot z+\fracj{\partial {\mathcal M}_z}{\partial \theta}\theta+\fracj{\partial {\mathcal M}_z}{\partial \dot \theta}\dot \theta+\fracj{\partial{\mathcal M}_z}{\partial \delta_r}\delta_r+\fracj{\partial {\mathcal M}_z}{\partial \delta_f}\delta_f,
 \end{array}}\eeq
 with the following values (for memory $\fracj{\partial c_z}{\partial \beta}(\beta_0)=\fracj{\partial c_z}{\partial \alpha}(\alpha_0)=18/\pi$):
 \beq\label{eq90}\hskip-.4cm{\begin{array}{l}\fracj{\partial {\mathcal F}_z}{\partial z}=0,\;\fracj{\partial {\mathcal M}_y}{\partial z}=0,\;\;

 \fracj{\partial {\mathcal F}_z}{\partial \delta_r}=\fracj{\varrho V^2}{2}S_r\fracj{\partial c_z}{\partial \beta}(\beta_0),\;\; \fracj{\partial {\mathcal F}_z}{\partial \delta_f}=\fracj{\varrho V^2}{2}S_f\fracj{\partial c_z}{\partial \alpha}(\alpha_0),
 \\\\
 \fracj{\partial {\mathcal F}_z}{\partial \dot z}=-\fracj{\varrho V}{2}{\Big [}S_r\fracj{\partial c_z}{\partial \beta}(\beta_0)+ S_f\fracj{\partial c_z}{\partial \alpha}(\alpha_0){\Big ]},
 \\\\
 \fracj{\partial {\mathcal M}_y}{\partial \dot z}=\fracj{\varrho V}{2}{\Big [}-S_rd_r\fracj{\partial c_z}{\partial \beta}(\beta_0)+S_fd_f\fracj{\partial c_z}{\partial \alpha}(\alpha_0){\Big]},
 ,\\\\
  \fracj{\partial {\mathcal F}_z}{\partial \theta}=\fracj{\varrho V^2}{2}{\Big [}S_r\fracj{\partial c_z}{\partial \beta}(\beta_0)+S_f\fracj{\partial c_z}{\partial \alpha}(\alpha_0){\Big ]},\\\\
  \fracj{\partial {\mathcal M}_y}{\partial \theta}=\fracj{\varrho V^2}{2}{\Big[}-S_rd_r\fracj{\partial c_z}{\partial \beta}(\beta_0)+S_fd_f\fracj{\partial c_z}{\partial \alpha}(\alpha_0){\Big ]}.
 \\\\
 \fracj{\partial {\mathcal F}_z}{\partial \dot \theta}=\varrho V{\Big [}S_rh_rc_z(\beta_0)+S_fh_fc_z(\alpha_0)+\fracj{S_rd_r}{2}\fracj{\partial c_z}{\partial \beta}(\beta_0)-\fracj{S_fd_f}{2}\fracj{\partial c_z}{\partial \alpha}(\alpha_0){\Big ]},
 \\\\
  \fracj{\partial {\mathcal M}_y}{\partial \dot \theta}\hskip-.1cm=\varrho V{\Big [} S_fd_fh_fc_z(\alpha_0)-S_rd_rh_rc_z(\beta_0)-\fracj{S_rd_r^2}{2}\fracj{\partial c_z}{\partial \beta}(\beta_0)-\fracj{S_fd_f^2}{2}\fracj{\partial c_z}{\partial \alpha}(\alpha_0){\Big ]}. \end{array}}\eeq
 We now introduce four matrices in order to define the dynamical model. One denoted by ${\mathcal M}$, is the inertia, the second one ${\mathcal C}$ couples the gyroscopic effect (odd part of ${\mathcal C}$) and the hydrodynamic damping (the symmetrical part of ${\mathcal C}$ which is not necessarily positive), the third one -say ${\mathcal K}$- is the stiffness (not necessarily neither symmetrical nor positive) and the last one ${\mathcal B}$ is the so-called {\it control tuner}.
 \\
 \beq\label{eq100}{\begin{array}{l}\begin{array}{cc}
 {\mathcal M}=\left (\begin{array}{l}\begin{array}{cc}M&0
 \\
 0&J\end{array}\end{array}\right )
 &
  {\mathcal C}=\left (\begin{array}{l}\begin{array}{cc}-\fracj{\partial {\mathcal F}_z}{\partial \dot z}&-\fracj{\partial {\mathcal F}_z}{\partial \dot \theta}\\\\-\fracj{\partial {\mathcal M}_y}{\partial \dot z}&-\fracj{\partial {\mathcal M}_y}{\partial \dot \theta}\end{array}\end{array}\right )\end{array}
  \\
  \begin{array}{cc}{\mathcal K}=\left (\begin{array}{l}\begin{array}{cc}-\fracj{\partial {\mathcal F}_z}{\partial  z}&-\fracj{\partial {\mathcal F}_z}{\partial \theta}\\\\-\fracj{\partial {\mathcal M}_y}{\partial  z}&-\fracj{\partial {\mathcal M}_y}{\partial \theta}\end{array}\end{array}\right )
  &
  \begin{array}{cc} {\mathcal B}=\left (\begin{array}{l}\begin{array}{cc}\fracj{\partial {\mathcal F}_z}{\partial  \delta_r}&\fracj{\partial {\mathcal F}_z}{\partial \delta_f}\\\\\fracj{\partial {\mathcal M}_y}{\partial  \delta_r}&\fracj{\partial {\mathcal M}_y}{\partial \delta_f}\end{array}\end{array}\right )\end{array}\end{array}\end{array}}
 \eeq
  The vectors representing the state variables, the control and the external perturbations are denoted by:
 $$ \hskip-.3cm\begin{array}{cccc}\hbox{Degrees of freedom: } X \hskip-.1cm= \hskip-.1cm\left (\begin{array}{l}z
 \\\\
 \theta
 \end{array}\right ),&
  \hskip-.2cm\hbox{Control: }
 u \hskip-.1cm= \hskip-.1cm \left (\begin{array}{l}u_r \\\\ u_f
 \end{array}\right )&
  \hskip-.2cm\hbox{Perturbations: }
 {\mathcal F} \hskip-.1cm= \hskip-.1cm\left (\begin{array}{l}f_1\\\\ f_2\end{array}\right )\end{array}$$
 \\
 and all the components of these vectors depend on time. Finally the equation of the movement is:
 \beq\label{eq100bis}{\mathcal M}\ddot X+{\mathcal C}\dot X+{\mathcal K}X={\mathcal B}u+{\mathcal F},\eeq
 with initial conditions on $X(0$ and $\dot X(0)$ which represent for instance the effect of a large wave (on $\dot \theta$...) or a brusk change in the wind direction or its intensity (on $\dot z$...).
 \end{document}